%% file: openclosedBPS.tex
\documentclass{amsart}
\usepackage[letterpaper, margin=1.4in]{geometry}
\usepackage{amssymb, latexsym, MnSymbol}
\usepackage{mathtools}
\usepackage{color, xcolor}
\usepackage{amscd, graphicx, psfrag, tikz}
\usepackage[mathscr]{euscript}
\usepackage[all]{xy}

\usepackage{enumitem}
\usepackage{stmaryrd}
\usepackage{url}
\usepackage{hyperref}
\hypersetup{
    colorlinks=true,
    linkcolor=black,
    citecolor=black,
    urlcolor=black
}
\usepackage{verbatim}

\newcommand{\bC}{ {\mathbb{C}} }

\newcommand{\bP}{\mathbb{P}}
\newcommand{\bQ}{\mathbb{Q}}
\newcommand{\bR}{\mathbb{R}}

\newcommand{\bZ}{\mathbb{Z}}

\newcommand{\cA}{\mathcal{A}}

\newcommand{\cD}{\mathcal{D}}
\newcommand{\cE}{\mathcal{E}}

\newcommand{\cH}{\mathcal{H}}

\newcommand{\cL}{\mathcal{L}}

\newcommand{\cO}{\mathcal{O}}
\newcommand{\cP}{\mathcal{P}}

\newcommand{\cW}{\mathcal{W}}

\newcommand{\Aut}{\mathrm{Aut}}

\newcommand{\Hom}{\mathrm{Hom}}

\newcommand{\pt}{\mathrm{pt}}

\newcommand{\Eff}{{\mathrm{Eff}}}

\newcommand{\Tot}{\mathrm{Tot}}
\newcommand{\sgn}{\mathrm{sgn}}
\newcommand{\Comb}{\mathrm{Comb}}

\newcommand{\bw}{\mathbf{w}}

\newcommand{\ff}{\mathfrak{f}}

\newcommand{\fl}{\mathfrak{l}}

\newcommand{\fp}{\mathfrak{p}}

\newcommand{\su}{\mathsf{u}}

\newcommand{\hD}{\hat{D}}
\newcommand{\hX}{\hat{X}}
\newcommand{\hY}{\hat{Y}}

\newcommand{\hSi}{\hat{\Sigma}}

\newcommand{\tSi}{\widetilde{\Sigma}}

\newcommand{\tC}{\widetilde{C}}
\newcommand{\tD}{\widetilde{D}}
\newcommand{\tF}{\widetilde{F}}

\newcommand{\tN}{\widetilde{N}}

\newcommand{\tT}{\widetilde{T}}

\newcommand{\tX}{\widetilde{X}}

\newcommand{\tb}{\widetilde{b}}

\newcommand{\tv}{\widetilde{v}}

\newcommand{\tgamma}{\widetilde{\gamma}}

\newcommand{\tbeta}{\widetilde{\beta}}

\newcommand{\va}{\vec{a}}
\newcommand{\vd}{\vec{d}}
\newcommand{\vf}{\vec{f}}

\newcommand{\vn}{\vec{n}}

\newcommand{\vmu}{\vec{\mu}}
\newcommand{\vnu}{\vec{\nu}}
\newcommand{\vdelta}{\vec{\delta}}
\newcommand{\vlambda}{\vec{\lambda}}
\newcommand{\vrho}{\vec{\rho}}
\newcommand{\vxi}{\vec{\xi}}

\newcommand{\inner}[1]{\langle  #1 \rangle}

\newtheorem{dummy}{dummy}[section]
\newtheorem{lemma}[dummy]{Lemma}
\newtheorem{theorem}[dummy]{Theorem}
\newtheorem{conjecture}[dummy]{Conjecture}
\newtheorem{corollary}[dummy]{Corollary}

\newtheorem{remark}[dummy]{Remark}
\newtheorem{definition}[dummy]{Definition}
\newtheorem{example}[dummy]{Example}

\newtheorem{assumption}[dummy]{Assumption}

\begin{document}
\title{Open/Closed BPS Correspondence and Integrality}

\author{Song Yu}
\address{Department of Mathematics, California Institute of Technology, Pasadena, CA, USA \vspace{-7pt}}
\address{Yau Mathematical Sciences Center, Tsinghua University, Beijing, China}
\email{songyu@caltech.edu}

\begin{abstract}
We prove the integrality and finiteness of open BPS invariants of toric Calabi-Yau 3-folds relative to Aganagic-Vafa outer branes, defined from open Gromov-Witten invariants by the Labastida-Mari\~no-Ooguri-Vafa formula. Specializing to disk invariants, we extend the open/closed correspondence of Gromov-Witten invariants to BPS invariants and prove the integrality of a class of genus-zero BPS invariants of toric Calabi-Yau 4-folds, thereby providing additional examples for the conjecture of Klemm-Pandharipande.
\end{abstract}



\maketitle

\setcounter{tocdepth}{1}
\tableofcontents

\input{intro}

\input{setup}

\input{integrality}

\input{correspondence}

\input{appdx}

\input{ref}
\end{document}

%% file: intro.tex
\section{Introduction}\label{sect:Intro}

\subsection{Open BPS invariants of Calabi-Yau 3-folds}\label{sect:IntroOpen}
Let $X$ be a smooth Calabi-Yau 3-fold. The famous conjecture of Gopakumar and Vafa \cite{GV98a,GV98b} concerns the relation between two enumerative theories of $X$:
\begin{itemize}
    \item \emph{Gromov-Witten invariants} $N^X_{g,\beta}$, which are virtual counts of curves in $X$ and are rational numbers in general;
    \item \emph{BPS invariants} $n^X_{g,\beta}$, which are counts of BPS states supported on curves in $X$ and are integers according to physical interpretations and predictions. 
\end{itemize}
Here, the invariants are parameterized by the genus $g \in \bZ_{\ge 0}$ and class $\beta \in H_2(X; \bZ)$ of the curve. Gopakumar and Vafa conjectured that the two sets of invariants are related by the resummation formula
\begin{equation}\label{eqn:IntroGVResum}
    \sum_{g \in \bZ_{\ge 0}} N^{X}_{g, \beta} g_s^{2g-2}\\
    = \sum_{k \mid \beta} \sum_{g \in \bZ_{\ge 0}} \frac{n^{X}_{g,\frac{\beta}{k}}}{k} \left(2\sin \frac{kg_s}{2} \right)^{2g-2}
\end{equation}
where $g_s$ is a formal variable, $\beta \in H_2(X; \bZ)$ is a non-zero effective curve class, and for $k \in \bZ_{\ge 1}$ we say that $k \mid \beta$ if $\frac{\beta}{k}\in H_2(X; \bZ)$.

While the mathematical foundations of Gromov-Witten theory are relatively well-established, a rigorous mathematical definition of the BPS invariants is yet to be given and there have been rich developments in the literature; see e.g. \cite{HST01,Katz08,PT10,MT18}. If one takes \eqref{eqn:IntroGVResum} to be the definition, the Gopakumar-Vafa conjecture can be phrased as \emph{integrality} and \emph{finiteness} properties of the BPS invariants thereby defined \cite{GV98a,GV98b,BP01}. Namely, for any fixed $\beta \in H_2(X; \bZ)$, $n^X_{g,\beta}$ is an integer for all $g \in \bZ_{\ge 0}$ and is zero for $g \gg 0$. When the Calabi-Yau 3-fold $X$ is compact, the integrality and finiteness parts were proved by Ionel-Parker \cite{IP18} and Doan-Ionel-Walpuski \cite{DIW21} respectively using symplectic methods. When $X$ is toric Calabi-Yau (in which case it is non-compact), a proof was given in the earlier works of Peng \cite{Peng07} and Konishi \cite{Konishi06a,Konishi06b} based on the computation of all-genus Gromov-Witten invariants by the topological vertex \cite{AKMV03,LLLZ09} and the Gromov-Witten/Donaldson-Thomas correspondence \cite{MNOP06,MOOP11}. The conjecture has also been studied in the Fano case by Zinger \cite{Zinger11} and Doan-Walpuski \cite{DW19}.

Analogously, for open topological strings, the open Gromov-Witten invariants of $X$ with prescribed Lagrangian boundary conditions are also expected to carry integrality properties encoded by open BPS counts. We consider the case where $X$ is toric and the boundary condition is given by a disjoint union $L = L_1 \sqcup \cdots \sqcup L_s$ of $s$ special Lagrangian submanifolds called \emph{Aganagic-Vafa outer branes} \cite{AV00,AKV02,FL13} framed by integers $\vf = (f_1, \dots, f_s)$. Each $L_i$ is diffeomorphic to $S^1 \times \bC$ and invariant under a $U(1)^2$-action on $X$. Given genus $g \in \bZ_{\ge 0}$, effective class $\beta \in H_2(X,L;\bZ)$, and a sequence of partitions $\vmu = (\mu^1, \dots, \mu^s)$, the \emph{open Gromov-Witten invariant}
$$
    N^{X, L, \vf}_{g, \beta, \vmu}
$$
is a virtual count of genus-$g$, degree-$\beta$ bordered Riemann surfaces in $X$ whose winding profile at $L_i$ is determined by the partition $\mu^i$. These invariants are also defined and computed by the topological vertex. A particularly well-studied example of such open geometries is the resolved conifold $X = \Tot(\cO_{\bP^1}(-1) \oplus \cO_{\bP^1}(-1))$ relative to a single brane, whose open Gromov-Witten invariants conjecturally correspond to knot and link invariants originating from Chern-Simons gauge theory on $S^3$; see e.g. \cite{Witten89,Witten95,GV99,OV00,LM00,LMV00,RS01,MV02}. Labastida, Mari\~no, Ooguri, and Vafa (LMOV) \cite{OV00,LM00,LMV00,MV02} discovered a resummation formula, similar to \eqref{eqn:IntroGVResum}, that exhibit the integrality of these invariants. For a general open geometry $(X, L, \vf)$ and winding profile $\vmu$, the formula can be written as
\begin{equation}\label{eqn:IntroLMOVResum}
    \begin{aligned}
        \sum_{g \in \bZ_{\ge 0}} & N^{X, L, \vf}_{g, \beta, \vmu} g_s^{2g-2+\ell(\vmu)} \\
        & = \frac{1}{\prod_{i = 1}^s z_{\mu^i}} \sum_{k \mid \beta, \vmu}  \sum_{g \in \bZ_{\ge 0}} (-1)^{\ell(\vmu)+g}k^{\ell(\vmu)-1}n^{X, L, \vf}_{g,\frac{\beta}{k}, \frac{\vmu}{k}}\left(2\sin \frac{kg_s}{2} \right)^{2g-2} \prod_{i=1}^s \prod_{j=1}^{\ell(\mu^i)} 2\sin \frac{\mu^i_jg_s}{2}.
    \end{aligned}
\end{equation}
Here, for each partition $\mu^i$, $\ell(\mu^i)$ is its length, the $\mu^i_j$'s are its parts, and
$$
    z_{\mu^i} := |\Aut(\mu^i)| \prod_{j=1}^{\ell(\mu^i)} \mu^i_j
$$
where $\Aut(\mu^i)$ is the automorphism group of $\mu^i$. We set $\ell(\vmu) := \sum_{i=1}^s \ell(\mu^i)$. Moreover, for $k \in \bZ_{\ge 1}$ we say that $k \mid \vmu$ if $k \mid \mu^i_j$ for all $i, j$, in which case $\frac{\vmu}{k}$ denotes the sequence of partitions whose parts are $\frac{\mu^i_j}{k}$.

We refer to the invariants
$$
    n^{X, L, \vf}_{g, \beta, \vmu}
$$
arising from \eqref{eqn:IntroLMOVResum} as the \emph{open BPS invariants} of $(X, L, \vf)$. Similar to the Gopakumar-Vafa conjecture in the closed sector, the open BPS invariants are also predicted to satisfy integrality and finiteness properties. As the first main result of the present paper, we verify this prediction.

\begin{theorem}[See Theorem \ref{thm:OpenBPS}]\label{thm:IntroOpenBPS}
For any effective class $\beta \in H_2(X,L;\bZ)$ and winding profile $\vmu$, we have that
$$
    n^{X, L, \vf}_{g, \beta, \vmu} \in \bZ
$$
for all $g \in \bZ_{\ge 0}$ and is zero for $g \gg 0$.
\end{theorem}

Theorem \ref{thm:IntroOpenBPS} generalizes several interesting special cases previously known in the literature. Luo and Zhu \cite{LZ16,LZ19} considered the resolved conifold relative to a single brane and studied its genus-zero invariants with a general winding profile and higher-genus invariants ``with one hole'', i.e. the length of the partition is 1. It is worth noting that their method involves the computation of the Gromov-Witten invariants by the Bouchard-Klemm-Mari\~{n}o-Pasquetti \emph{Remodeling conjecture} \cite{BKMP09,BKMP10,EO15,FLZ20}. Zhu \cite{Zhu19} further studied higher-genus one-hole invariants of $\bC^3$ relative to a single brane. More recently, Panfil and Su\l{}kowski \cite{PS19} considered \emph{``strip geometries''}, which is a class of toric Calabi-Yau 3-folds without compact 4-cycles, and their \emph{disk invariants}, or genus-zero one-hole invariants, with boundary at a certain choice of a brane. Using the \emph{knots-quivers correspondence} \cite{KRSS17,KRSS19,EKL20a,EKL20b}, they related such invariants to the Donaldson-Thomas invariants of symmetric quivers, which are known to be integers \cite{Efimov12}. This perspective was taken by Bousseau, Brini, van Garrel \cite{BBvG20}, and Zhu \cite{Zhu22}. In particular, \cite{BBvG20} considered strip geometries that arise from \emph{Looijenga pairs}, or log Calabi-Yau surfaces with maximal boundary, and established an all-genus correspondence between the open and log Gromov-Witten invariants. They further proved the integrality and finiteness of the higher-genus open BPS invariants, which then carry over to the log side. Integrality properties of log invariants in connection with open invariants have also been studied in the recent works \cite{GRZ22,BS23,Schuler24}.

Compared to the previous works, Theorem \ref{thm:IntroOpenBPS} takes a very general form in the sense that it applies to a general toric Calabi-Yau 3-fold, any number of branes in an arbitrary configuration, and any winding profiles. In terms of techniques, the proof is closest to that of \cite{Peng07,Konishi06a,Konishi06b} in the closed sector which is based on the topological vertex.

\subsection{Closed BPS invariants of Calabi-Yau 4-folds and the open/closed correspondence}\label{sect:IntroClosed}
Back to the closed sector and in the genus zero case, the Gopakumar-Vafa formula \eqref{eqn:IntroGVResum} specializes to the Aspinwall-Morrison multiple covering formula \cite{AM93}
\begin{equation}\label{eqn:IntroAMResum}
    N^{X}_{0, \beta} \\
    = \sum_{k \mid \beta} \frac{n^{X}_{0,\frac{\beta}{k}}}{k^3}.
\end{equation}
This formula is also expected to exhibit the integrality of Gromov-Witten invariants of Calabi-Yau manifolds of higher dimensions. Klemm and Pandharipande \cite[Conjecture 0]{KP08} conjectured that for a Calabi-Yau 4-fold $Z$ and insertions $\gamma_1, \dots, \gamma_n \in H^*(Z;\bZ)$, the BPS invariants $n^{Z}_{0, \beta}(\gamma_1, \dots, \gamma_n)$ defined from the Gromov-Witten invariants $N^{Z}_{0, \beta}(\gamma_1, \dots, \gamma_n)$ by the formula
\begin{equation}\label{eqn:IntroKPResum}
    N^{Z}_{0, \beta}(\gamma_1, \dots, \gamma_n) \\
    = \sum_{k \mid \beta} \frac{n^{Z}_{0,\frac{\beta}{k}}(\gamma_1, \dots, \gamma_n)}{k^{3-n}}
\end{equation}
are integers. The conjecture has been verified in examples in \cite{KP08,CMT18,Cao20,CMT22,CKM22,COT22,COT24,BBvG20,BS23}. For a compact semi-positive symplectic manifold $Z$ of real dimension at least 6, the conjecture was proved by Ionel-Parker \cite{IP18}. We note that the approach of \cite{CMT18,Cao20,CMT22,CKM22,COT22,COT24} is based on the proposal that the BPS invariants of Calabi-Yau 4-folds admit sheaf-counting interpretations in terms of e.g. Donaldson-Thomas invariants and Pandharipande-Thomas invariants. We do not pursue this perspective in the present paper.

On the other hand, we take the approach of relating closed invariants of 4-folds to open invariants of 3-folds via the \emph{open/closed correspondence}, developed by Liu and the author \cite{LY21,LY22} based on the original proposal of Mayr \cite{Mayr01} and recent studies of correspondences among different types (open, log/relative, local) of enumerative invariants \cite{LLLZ09,FL13,vGGR19,BBvG20}. To be more specific, given an open geometry on a toric Calabi-Yau 3-fold $X$ relative to a single framed outer brane $(L,f)$, we consider its disk invariants $N^{X, L, f}_{0, \beta, (d)}$ where for an effective class $\beta \in H_2(X,L;\bZ)$, $d \in \bZ_{\ge 1}$ denotes the winding number at the unique boundary component of the curve and is determined by $\beta$. Assuming that $L$ remains an outer brane in a toric Calabi-Yau semi-projective partial compactification of $X$ (see Assumption \ref{assump:Outer}), on the level of Gromov-Witten theory, the open/closed correspondence produces a toric Calabi-Yau 4-fold $\tX$ together with an isomorphism $\iota: H_2(X, L;\bZ) \to H_2(\tX;\bZ)$ and a class $\tgamma \in H^4(\tX;\bZ)$ such that
$$
    N^{X, L, f}_{0, \beta, (d)} = N^{\tX}_{0, \iota(\beta)}(\tgamma)
$$
for any $\beta$.

Here we give a high-level summary of the construction of $\tX$, deferring a more detailed description to Section \ref{sect:OpenClosedGW}. Starting from the open geometry $(X,L,f)$, one first partially compactifies $X$ by adding an additional toric divisor $D$ at the location of the brane $L$, obtaining a toric 3-fold $Y = X \sqcup D$ such that the pair $(Y,D)$ is log Calabi-Yau. This step is the construction of \cite{LLLZ09,FL13} who identified the open Gromov-Witten invariants of $(X,L,f)$ with the relative invariants of $(Y,D)$. Then, $\tX$ is taken to be the total space of the canonical bundle $\cO_Y(-D)$, which is a toric Calabi-Yau 4-fold. This may be viewed as an instantiation of the \emph{log-local principle} \cite{vGGR19} which identifies the log/relative invariants of $(Y,D)$ with the closed invariants of the local geometry $\tX$. We note that when $X$ is semi-projective, \cite{LY22} further established the open/closed correspondence with $\tX$ replaced by its toric Calabi-Yau semi-projective partial compactification. 

Now in the case of disk invariants, the LMOV formula \eqref{eqn:IntroLMOVResum} specializes to
$$
    N^{X, L, f}_{0, \beta, (d)} = - \sum_{k \mid \beta, d} \frac{n^{X, L, f}_{0,\frac{\beta}{k}, \left(\frac{d}{k}\right)}}{k^2},
$$
which coincides up to sign with the Klemm-Pandharipande resummation \eqref{eqn:IntroKPResum} for invariants with 1-pointed insertions. Indeed, this was already observed by Bousseau, Brini, and van Garrel \cite{BBvG20} who obtained the integrality of BPS invariants of closed geometries arising from Looijenga pairs. We establish the open/closed correspondence for the BPS invariants of more general geometries.

\begin{theorem}[See Theorem \ref{thm:BPSCorrespondence}]\label{thm:IntroBPSCorrespondence}
For the open geometry $(X,L,f)$ and closed geometry $\tX$ under the open/closed correspondence above, given any effective class $\beta \in H_2(X,L;\bZ)$, we have
$$
    n^{X, L, f}_{0, \beta, (d)} = -n^{\tX}_{0, \iota(\beta)}(\tgamma).
$$
\end{theorem}

Combining this result with Theorem \ref{thm:IntroOpenBPS}, we obtain the integrality of the closed BPS invariants of $\tX$.

\begin{corollary}[See Corollary \ref{cor:ClosedBPS}]\label{cor:IntroClosedBPS}
Under the setup of Theorem \ref{thm:IntroBPSCorrespondence}, we have
$$
    n^{\tX}_{0, \iota(\beta)}(\tgamma) \in \bZ.
$$
\end{corollary}

Therefore, we obtain a general class of non-compact examples for the conjecture of Klemm-Pandharipande \cite[Conjecture 0]{KP08} which covers certain previously known examples. We give a more detailed account in Section \ref{sect:ClosedBPS}.


\subsection{Conjectures on open BPS invariants of general geometries}
With the toric case as evidence, it would be interesting to study the integrality of open BPS invariants of more general geometries. Let $(X,\omega,J)$ be a Calabi-Yau symplectic 6-manifold together with an almost complex structure $J$ compatible with the symplectic form $\omega$, and $L \subset X$ be a Lagrangian submanifold that is orientable and has zero Maslov class. For simplicity, suppose $L$ is connected. In this case, the moduli space of $J$-holomorphic open stable maps to $(X,L)$ has virtual dimension 0 regardless of the genus or number of boundary components of the domain \cite{KL01}, and one may attempt to define all-genus open Gromov-Witten invariants of $(X,L)$, e.g. when $X$ is compact or $(X,L)$ admits a suitable $U(1)$-action. The open BPS invariants of $(X,L)$ may then be defined by the LMOV formula \eqref{eqn:IntroLMOVResum} and its generalizations.

We first make the following proposal for disk invariants.

\begin{conjecture}
Assume that $H_2(X,L;\bZ)$ is torsion free and the disk invariant $N^{X,L}_{0,\beta}$ is defined for any effective class $\beta \in H_2(X,L;\bZ)$. Then the open BPS invariants $n^{X,L}_{0,\beta}$ defined by the formula
$$
    N^{X, L}_{0, \beta} = - \sum_{k \mid \beta} \frac{n^{X, L}_{0,\frac{\beta}{k}}}{k^2}
$$
are integers for all $\beta$.
\end{conjecture}

We refer to \cite{PSW08} and the references therein for discussions on the case where $X$ is the quintic 3-fold and $L \subset X$ is the real quintic. The open/closed correspondence has also been extended to this geometry \cite{AL23}.

In addition, in the case where a notion of winding numbers can be defined, we make the following proposal.

\begin{conjecture}\label{conj:General}
Assume that $H_2(X,L;\bZ)$ is torsion free, $H_1(L;\bZ) \cong \bZ$, and the open Gromov-Witten invariant $N^{X,L}_{0,\beta,\mu}$ is defined for any genus $g \in \bZ_{\ge 0}$, effective class $\beta \in H_2(X,L;\bZ)$, and partition $\mu$ specifying the winding profile. Let the open BPS invariants $n^{X,L}_{g,\beta, \mu}$ be defined by the formula
$$    
    \sum_{g \in \bZ_{\ge 0}} N^{X, L}_{g, \beta, \mu} g_s^{2g-2+\ell(\mu)} 
    = \frac{1}{z_{\mu}} \sum_{k \mid \beta, \mu}  \sum_{g \in \bZ_{\ge 0}} (-1)^{\ell(\mu)+g}k^{\ell(\mu)-1}n^{X, L}_{g,\frac{\beta}{k}, \frac{\mu}{k}}\left(2\sin \frac{kg_s}{2} \right)^{2g-2} \prod_{j=1}^{\ell(\mu)} 2\sin \frac{\mu_jg_s}{2}.
$$
Then for any $\beta, \mu$, we have that $n^{X, L}_{g, \beta, \mu} \in \bZ$ for all $g \in \bZ_{\ge 0}$ and is zero for $g \gg 0$.
\end{conjecture}

It would be interesting to generalize the LMOV formula \eqref{eqn:IntroLMOVResum} to the case where $H_1(L;\bZ)$ has higher rank and study the integrality and finiteness of open BPS invariants in the general setup. In any case, for a general geometry $(X,L)$, the LMOV formula can be employed to study the \emph{local} contributions of embedded (bordered) curves and their multiple covers, following the approach of \cite{IP18,DIW21}. To be more specific, let $(C, \partial C) \subset (X, L)$ be a bordered embedded curve and consider the pair $(E, E_{\bR})$, where $E := N_{C/X}$ and $E_{\bR} := N_{\partial C/L}$ which is a Lagrangian (real) subbundle of $N_{C/X} \big|_{\partial C}$ (c.f. \cite{KL01}). As in Bryan-Pandharipande \cite{BP08} and Ionel-Parker \cite{IP18}, one may consider the local Calabi-Yau geometry on the total space of $(E, E_{\bR})$, define a local (or residue) version of open Gromov-Witten invariants using a suitable fiberwise $U(1)$-action, define a local version of open BPS invariants using the LMOV formula, and study their integrality and finiteness as in Conjecture \ref{conj:General}. This would account for the multiple-cover contributions of $(C, \partial C)$. Under the philosophy of \cite{IP18,DIW21}, the global statement for a relative curve class $\beta \in H_2(X,L;\bZ)$ would follow from collecting contributions from a finite set of embedded curves in this class.

\subsection{Acknowledgments}
This work originates from the author's joint project with Chiu-Chu Melissa Liu \cite{LY21,LY22} and I would like to thank Melissa for many valuable discussions and suggestions. I would like to thank Eleny Ionel for enlightening discussions on the study of open BPS invariants of general geometries. I would like to thank Bohan Fang, Yannik Schuler, and Zhengyu Zong for helpful discussions. I would also like to thank the authors of \cite{BBvG20} for their inspiring results on the open and closed BPS invariants. Finally, I would like to thank the anonymous referees for the valuable comments which greatly helped improve the paper.

%% file: setup.tex
\section{Open Gromov-Witten invariants and the topological vertex}\label{sect:OpenGW}
In this section, we briefly review the open geometry setup, open Gromov-Witten invariants, and the topological vertex formalism, introducing notations along the way. We refer to \cite{LLLZ09,FL13} for full definitions and additional details. We work over $\bC$.

\subsection{Preliminaries}\label{sect:Partition}
We start with the preliminaries on partitions. 
\subsubsection{Notations on a single partition}
A \emph{partition} is a non-increasing sequence
$$
    \lambda = (\lambda_1, \lambda_2, \dots)
$$
of non-negative integers where only finitely many are non-zero. We set up the following notations on $\lambda$:

\begin{itemize}
    \item The \emph{length} of $\lambda$ is the number of non-zero terms, called \emph{parts} of $\lambda$:
    $$
        \ell(\lambda) := |\{j \mid \lambda_j \neq 0\}|.
    $$
    We will interchangeably write $\lambda$ as the finite sequence consisting of its parts:
        $$
        \lambda = (\lambda_1, \dots, \lambda_{\ell(\lambda)})
        $$
    
    \item The \emph{degree} of $\lambda$ is the sum of all parts:
        $$
            |\lambda| := \sum_{j=1}^{\ell(\lambda)} \lambda_j.
        $$
    If $d = |\lambda|$, we say that $\lambda$ is a partition of $d$ and write $\lambda \vdash d$.

    \item For $a \in \bZ_{\ge 1}$, the \emph{multiplicity} of $a$ in $\lambda$ is the number of occurrences of $a$ in $\lambda$:
        $$
            m_a(\lambda) := |\{j \mid \lambda_j = a\}|.
        $$
    
    \item The \emph{automorphism group} of $\lambda$, denoted $\Aut(\lambda)$, is the group of permutations of parts of $\lambda$ that preserve $\lambda$. We have
        $$
            \Aut(\lambda) \cong \prod_{a \in \bZ_{\ge 1}} S_{m_a(\lambda)}, \qquad |\Aut(\lambda)| = \prod_{a \in \bZ_{\ge 1}} m_a(\lambda)!
        $$
    where $S_m$ denotes the symmetric group on $m$ letters.

    \item We set
        $$
            z_\lambda := |\Aut(\lambda)| \prod_{j=1}^{\ell(\lambda)} \lambda_j.
        $$

    \item The \emph{conjugate partition} of $\lambda$ is the sequence $\lambda^t = (\lambda^t_1, \lambda^t_2, \dots)$ where
        $$
            \lambda^t_i := \sum_{a \in \bZ_{\ge i}} m_a(\lambda).
        $$

    \item We set
        $$
            \kappa_\lambda := \sum_{j=1}^{\ell(\lambda)} \lambda_j(\lambda_j - 2j + 1),
        $$
    which is even and satisfies
        $$
            \kappa_{\lambda^t} = - \kappa_\lambda.
        $$

    \item For $k \in \bZ_{\ge 1}$, we set
        $$
            k\lambda := (k\lambda_1, \dots, k\lambda_{\ell(\lambda)}).
        $$
    We have
        $$
            \ell(k\lambda) = \ell(\lambda), \quad |k\lambda| = k |\lambda|, \quad \Aut(k\lambda) \cong \Aut(\lambda), \quad z_{k\lambda} = k^{\ell(\lambda)}z_{\lambda}.
        $$
    On the other hand, if there exists a partition $\mu$ such that $\lambda = k\mu$, we write $k \mid \lambda$ and $\mu = \frac{\lambda}{k}$.
    
    \item For $k \in \bZ_{\ge 1}$, let $\lambda^{(k)}$ denote the partition where
        $$
            m_a(\lambda^{(k)}) = k m_a(\lambda)
        $$
    for all $a \in \bZ_{\ge 1}$. On the other hand, if there exists a partition $\mu$ such that $\lambda = \mu^{(k)}$, we write $\mu = \lambda^{(\frac{1}{k})}$.
\end{itemize}

\subsubsection{Notations on multiple partitions}
Let $\cP$ denote the set of all partitions. Given $\lambda^1, \dots, \lambda^t \in \cP$, we define
$$
    \bigsqcup_{i = 1}^t \lambda^i
$$
to be the partition whose set of parts is the disjoint union of all sets of parts of $\lambda^1, \dots, \lambda^t$. In particular, when $\lambda^1 = \cdots = \lambda^t = \lambda$, $\bigsqcup_{i = 1}^t \lambda^i = \lambda^{(t)}$. We have
$$
    \ell\left( \bigsqcup_{i = 1}^t \lambda^i\right) = \sum_{i=1}^t \ell(\lambda^i), \qquad \left|\bigsqcup_{i = 1}^t \lambda^i\right| = \sum_{i=1}^t |\lambda^i|.
$$
Moreover, $\prod_{i=1}^t \Aut(\lambda^i)$ is a subgroup of $\Aut \left(\bigsqcup_{i = 1}^t \lambda^i \right)$ and
$$
    \prod_{i=1}^t z_{\lambda^i} \biggm| z_{\bigsqcup_{i = 1}^t \lambda^i}.
$$
For partitions $\mu, \nu \in \cP$, we write
$$
    \nu \subseteq \mu
$$
if there exists $\nu' \in \cP$ such that $\mu = \nu \sqcup \nu'$, or equivalently $m_a(\nu) \le m_a(\mu)$ for all $a \in \bZ_{\ge 1}$.

Now given a sequence of partitions $\vmu = (\mu^1, \dots, \mu^s) \in \cP^s$, we denote
$$
    \ell(\vmu) := \sum_{i = 1}^s \ell(\mu^i), \qquad |\vmu| := \sum_{i=1}^s |\mu^i|, \qquad z_{\vmu}:= \prod_{i=1}^s z_{\mu^i}.
$$
For $k \in \bZ_{\ge 1}$, we define
$$
    k\vmu := (k\mu^1, \dots, k\mu^s), \qquad \vmu^{(k)} := ((\mu^1)^{(k)}, \dots, (\mu^s)^{(k)}),
$$
and $\frac{\vmu}{k}$ (resp. $\vmu^{(\frac{1}{k})}$) similarly if $\frac{\mu^i}{k}$ (resp. $(\mu^i)^{(\frac{1}{k})}$) exists for all $i$. Now for another sequence $\vnu = (\nu^1, \dots, \nu^s) \in \cP^s$, we define
$$
    \vmu \sqcup \vnu := (\mu^1 \sqcup \nu^1, \dots, \mu^s \sqcup \nu^s),
$$
and we write
$$
    \vnu \subseteq \vmu
$$
if $\nu^i \subseteq \mu^i$ for all $i$.

\subsubsection{Three point function}
Finally, we define a function associated to triples of partitions following \cite{AKMV03, Zhou03, LLLZ09, Konishi06b}. Let $q$ be a formal variable and write
$$
    q^{\rho} = \left(q^{-\frac{1}{2}}, q^{-\frac{3}{2}}, \cdots \right), \qquad
    q^{\lambda + \rho} =  \left(q^{\lambda_1-\frac{1}{2}}, q^{\lambda_2-\frac{3}{2}}, \cdots \right) \quad \text{for $\lambda \in \cP$.}
$$
For $\lambda, \eta \in \cP$, let $s_{\lambda}, s_{\lambda/\eta}$ denote the Schur function and skew Schur function respectively.

\begin{definition}\label{def:ThreePoint}
\rm{
For a triple $(\lambda^1, \lambda^2, \lambda^3) \in \cP^3$ of partitions, define
$$
    \cW_{\lambda^1, \lambda^2, \lambda^3}(q) := q^{\frac{\kappa_{\lambda^3}}{2}}s_{\lambda^2}(q^\rho) \sum_{\eta \in \cP} s_{\lambda^1/\eta}(q^{(\lambda^2)^t+\rho})s_{(\lambda^3)^t/\eta}(q^{\lambda^2+\rho}).
$$
}
\end{definition}

Our expression is the same as \cite[Definition 2.2]{Konishi06b} and is equivalent to that of \cite{AKMV03,Zhou03,LLLZ09} (see \cite[Definition 2.1]{LLLZ09}) by \cite[Proposition 4.4]{Zhou03}. $\cW_{\lambda^1, \lambda^2, \lambda^3}(q)$ is a rational function in $q^{\frac{1}{2}}$ and has cyclic symmetry \cite{ORV03}
$$
    \cW_{\lambda^1, \lambda^2, \lambda^3} = \cW_{\lambda^2, \lambda^3, \lambda^1} = \cW_{\lambda^3, \lambda^1, \lambda^2}.
$$
We refer to \cite{Zhou03} for additional properties.

\subsection{Open geometry}\label{sect:OpenGeometry}
Let $N \cong \bZ^3$ and $X$ be a smooth toric Calabi-Yau 3-fold defined by a finite fan $\Sigma$ in $N \otimes \bR \cong \bR^3$. We assume that $\Sigma$ contains at least one 3-cone and that every cone in $\Sigma$ is a face of a 3-cone. The Calabi-Yau condition, i.e. the triviality of the canonical bundle $K_X$ of $X$, is equivalent to the existence of $\su_3 \in M := \Hom(N, \bZ)$ such that
$$
    \inner{\su_3, v} = 1
$$
for all primitive generators $v \in N$ of rays of $\Sigma$, where $\inner{-,-}: M \times N \to \bZ$ is the natural pairing. Let $N' := \ker(\su_3) \cong \bZ^2 \subset N$. The algebraic 3-torus of the toric 3-fold $X$ is $T := N \otimes \bC^* \cong (\bC^*)^3$, which contains the Calabi-Yau 2-subtorus $T' := N' \otimes \bC^* \cong (\bC^*)^2$. We complete $\su_3$ into a $\bZ$-basis $\{\su_1, \su_2, \su_3\}$ of $M$. Then
$$
    H^*_{T'}(\pt;\bZ) \cong \bZ[\su_1, \su_2].
$$

We set up some notations. For $d = 0,1,2,3$, let $\Sigma(d)$ to denote the set of $d$-cones in $\Sigma$. For a cone $\sigma \in \Sigma(d)$, let $V(\sigma)$ denote the codimension-$d$ $T$-orbit closure in $X$. In particular, the set $\Sigma(3)$ of maximal cones corresponds to the set of $T$-fixed points, and $\Sigma(2)$ corresponds to the set of $T$-invariant lines, which are isomorphic to either $\bC$ or $\bP^1$. We note that the action of the subtorus $T'$ has the same sets of fixed points and invariant lines. Let
$$
    \Sigma(2)_c := \{\tau \in \Sigma(2) \mid V(\tau) \cong \bP^1\}.
$$
We define the 1-skeleta
$$
    X^1 := \bigcup_{\tau \in \Sigma(2)} V(\tau), \qquad X^1_c := \bigcup_{\tau \in \Sigma(2)_c} V(\tau) \subset X^1.
$$
The inclusion $X^1_c \to X^1 \to X$ induces a surjective group homomorphism
\begin{equation}\label{eqn:CurveClass}
    H_2(X^1_c;\bZ) = H_2(X^1;\bZ) = \bigoplus_{\tau \in \Sigma(2)_c} \bZ[V(\tau)] \to H_2(X;\bZ).
\end{equation}

Let $s \in \bZ_{\ge 1}$ and 
$$
    L = L_1 \sqcup \cdots \sqcup L_s
$$
be a disjoint union of $s$ \emph{Aganagic-Vafa outer branes} in $X$ (see \cite[Section 2]{FL13}). Each $L_i$ is a Lagrangian submanifold of $X$ diffeomorphic to $S^1 \times \bC$ and invariant under the action of the maximal compact subtorus $T'_{\bR} \cong U(1)^2$ of $T'$. Moreover, each $L_i$ intersects a unique 1-dimensional $T$-orbit closure $V(\tau_i)$ for some $\tau_i \in \Sigma(2) \setminus \Sigma(2)_c$. The intersection $L_i \cap V(\tau_i)$ is isomorphic to $S^1$ and bounds a holomorphic disk $B_i$ in $V(\tau_i) \cong \bC$ centered around the $T$-fixed point $V(\sigma_i)$, where $\sigma_i \in \Sigma(3)$ is the unique 3-cone that contains $\tau_i$. We orient the disks $B_i$'s by the holomorphic structure of $X$. We assume that $\tau_1, \dots, \tau_s$ are distinct. Then
\begin{equation}\label{eqn:RelHomology}
    H_2(X,L;\bZ) = H_2(X; \bZ) \oplus \bigoplus_{i=1}^s \bZ[B_i], \qquad H_1(L; \bZ) =  \bigoplus_{i=1}^s \bZ[\partial B_i].
\end{equation}
Moreover, we choose additional parameters
$$
    \vf = (f_1, \dots, f_s)
$$
where for $i = 1, \dots, s$, $f_i \in \bZ$ is called the \emph{framing} of the brane $L_i$.

\subsection{FTCY graphs}\label{sect:FTCY}
The toric Calabi-Yau 3-fold $X$ can be equivalently described by its associated \emph{formal toric Calabi-Yau (FTCY) graph} $\Gamma_X$ as in \cite[Section 3]{LLLZ09}. The underlying graph of $\Gamma_X$ is a trivalent graph where:
\begin{itemize}
    \item The set $V(\Gamma_X)$ of vertices corresponds to the set $\Sigma(3)$ of 3-cones, or equivalently the set of $T$-fixed points. For $v \in V(\Gamma_X)$ we write $\sigma^v \in \Sigma(3)$ for the corresponding $3$-cone.
    
    \item The set $E(\Gamma_X)$ of (unoriented) edges corresponds to the set $\Sigma(2)$ of 2-cones, or equivalently the set of $T$-invariant lines. The subset $E_c(\Gamma_X)$ of compact edges corresponds to $\Sigma(2)_c$.\footnote{As in \cite[Section 3.1]{FL13}, we do not replace the non-compact edges by compact ones ending at univalent vertices, which is done in \cite{LLLZ09}.} For $\bar{e} \in E(\Gamma_X)$ we write $\tau^{\bar{e}} \in \Sigma(2)$ for the corresponding $2$-cone. 
    
    \item An edge $\bar{e} \in E(\Gamma_X)$ is incident to a vertex $v \in V(\Gamma_X)$ if and only if $\tau^{\bar{e}}$ is a facet of $\sigma^v$.
\end{itemize}
The data of $\Gamma_X$ also consists of a \emph{position map}
$$
    \fp: E^o(\Gamma_X) \to \bZ^2 \setminus \{0\},
$$
where $E^o(\Gamma_X)$ is the set of \emph{oriented} edges. To describe $\fp$, we may draw $\Gamma_X$ on the hyperplane $(N' \otimes \bR) \times \{1\}$ in $N \otimes \bR$ as a planar graph dual to the triangulated polytope given by the cross-section of the fan $\Sigma$. Under the coordinates on $M$ specified by the choice of basis $\{\su_1, \su_2, \su_3\}$, and thus the dual coordinates on $N$, $\fp(e)$ is given by the primitive integral vector in the direction of the oriented edge $e \in E^o(\Gamma_X)$ in the planar drawing.

As in \cite[Section 3.1]{FL13}, the framed Aganagic-Vafa branes $(L_1, f_1), \dots, (L_s,f_s)$ determine a FTCY graph
$$
    \Gamma = \Gamma_{X, L, \vf}
$$
obtained by replacing the non-compact edge in $\Gamma_X$ corresponding to $\tau_i$ by a compact edge $\bar{e}_i$ ending at a univalent vertex $v_i$ and with \emph{framing vector} $\ff_i \in \bZ^2 \setminus \{0\}$ depending on $f_i$. The set of vertices in $\Gamma$ is decomposed as
$$
    V(\Gamma) = V^3(\Gamma) \sqcup V^1(\Gamma)
$$
where $V^3(\Gamma) := V(\Gamma_X)$ is the set of trivalent vertices in $\Gamma$ and $V^1(\Gamma):= \{v_1, \dots, v_s\}$ is the set of univalent vertices. The set of (unoriented) compact edges in $\Gamma$ is decomposed as
$$
    E_c(\Gamma) = E_c^3(\Gamma) \sqcup E_c^1(\Gamma)
$$
where $E_c^3(\Gamma) := E_c(\Gamma_X)$ and $E_c^1(\Gamma) := \{\bar{e}_1, \dots, \bar{e}_s\}$.

As in \cite[Definition 3.4]{LLLZ09}, the position map $\fp$ and the framing vectors $\ff_i$'s together define a map
$$
    \vn: E^o(\Gamma) \to \bZ,
$$
where $E^o(\Gamma)$ is the set of oriented edges in $\Gamma$. We write $n^e := \vn(e)$. Let $E^o_c(\Gamma)$ denote the subset of oriented compact edges and for $e \in E^o_c(\Gamma)$, let $-e \in E^o_c(\Gamma)$ denote the edge with opposite orientation. Then $\vn$ satisfies that
$$
    n^{-e} = -n^e.
$$
For oriented edges $\pm e \in E^o_c(\Gamma)$ whose corresponding unoriented edge $\bar{e}$ belongs to $E_c^3(\Gamma)$ and corresponds to $\tau \in \Sigma(2)$, $n^{\pm e}$ is determined by the degree of the normal bundle of $V(\tau)$ in $X$:
$$
    N_{V(\tau)/X} \cong \cO_{\bP^1}(n^e-1) \oplus \cO_{\bP^1}(n^{-e}-1).
$$
For the oriented edge $e_i$ given by orienting $\bar{e}_i$ as terminating at $v_i$, we have
$$
    n^{e_i} = f_i.
$$

Let $\hX$ denote the FTCY 3-fold associated to the FTCY graph $\Gamma_X$, which is the formal completion of $X$ along the 1-skeleton $X^1$. Let $(\hY, \hD)$ denote the relative FTCY 3-fold associated to $\Gamma$, which satisfies $\hY \setminus \hD = \hX$ and admits an action of the Calabi-Yau 2-torus $T'$. The divisor $\hD$ is a disjoint union of its $s$ connected components $\hD^s, \dots, \hD^s$ corresponding to the $s$ univalent vertices $v_1, \dots, v_s \in V^1(\Gamma)$. Each edge $\bar{e}_i \in E^1_c(\Gamma)$ corresponds to a $T'$-invariant projective line $C_i$ in $\hY$ which is the compactification of $V(\tau_i) \subset \hX$ by the $T'$-fixed point contained in $\hD^i$. The homomorphism \eqref{eqn:CurveClass} extends to a surjective group homomorphism
\begin{equation}\label{eqn:RelCurveClass}
    \pi: H_2(\hY;\bZ) = H_2(X^1;\bZ) \oplus \bigoplus_{i=1}^s \bZ[C_i] \to H_2(X,L;\bZ) = H_2(X; \bZ) \oplus \bigoplus_{i=1}^s \bZ[B_i]
\end{equation}
that maps each $[C_i]$ to $[B_i]$ (see \eqref{eqn:RelHomology}).

\subsection{Gromov-Witten invariants}\label{sect:GW}
Let $g \in \bZ_{\ge 0}$. Let
$$
    \beta = \beta' + \sum_{i=1}^s d_i[B_i] \qquad \in H_2(X,L; \bZ)
$$
where $\beta' \in H_2(X;\bZ)$ is an effective class\footnote{Our notations $\beta, \beta'$ are interchanged from those in \cite{FL13}.} and $(d_1, \dots, d_s) \in \bZ_{\ge 0}^s \setminus \{(0, \dots, 0)\}$. Let
$$
    \vmu = (\mu^1, \dots, \mu^s) \in \cP^s \setminus \{(\emptyset, \dots, \emptyset)\}
$$
be partitions with $\mu^i \vdash d_i$ for $i = 1, \dots, s$, where $\emptyset$ denotes the empty partition.

\begin{definition}\label{def:OpenClass}
\rm{
The pair $(\beta, \vmu)$ as above is called an \emph{effective class} of $(X, L, \vf)$. Let
$$
    \Eff(X, L, \vf)
$$
denote the set of all effective classes.
}
\end{definition}

The genus-$g$, degree-$(\beta, \vmu)$ \emph{open Gromov-Witten invariant}
$$
    N^{X, L, \vf}_{g, \beta, \vmu} \in \bQ
$$
is a virtual count of open stable maps
$$
    u: \left(C, \partial C = \bigsqcup_{i = 1}^s \bigsqcup_{j = 1}^{\ell(\mu^i)} R^i_j \right) \to (X,L)
$$
where:
\begin{itemize}
    \item $C$ is a connected prestable bordered Riemann surface of arithmetic genus $g$ and whose boundary $\partial C$ consists of $\sum_{i=1}^s \ell(\mu^i)$ components $R^i_j \cong S^1$.
    
    \item $u_*[C] = \beta \in H_2(X,L;\bZ)$ and $u_*[R^i_j] = \mu^i_j[\partial B_i] \in H_1(L^i; \bZ)$.
\end{itemize}

Following \cite[Section 3.4]{FL13}, we define $N^{X, L, \vf}_{g, \beta, \vmu}$ by \emph{formal relative Gromov-Witten invariants} of the relative FTCY 3-fold $(\hY, \hD)$.\footnote{In the case $s=1$ of a single brane, see \cite[Section 3.5]{FL13} for an equivalent definition by virtual integration on the moduli spaces of open stable maps; see also \cite{KL01}.} Let
$$
    \vd: E_c(\Gamma) \to \bZ_{\ge 0}
$$
such that the image of the effective class
$$
    \sum_{\bar{e} \in E^3_c(\Gamma)} \vd(\bar{e})[V(\tau^{\bar{e}})] + \sum_{i = 1}^s \vd(\bar{e}_i)[C_i] \qquad \in H_2(\hY;\bZ),
$$
under the map $\pi$ (see \eqref{eqn:RelCurveClass}) is $\beta$.\footnote{Our notation $\vd$ corresponds to $\vd'$ in \cite{FL13}.} In particular, $\vd(\bar{e}_i) = d_i$ for $i = 1, \dots, s$. We denote this effective class of $\hY$ also by $\vd$ by an abuse of notation. 

\begin{definition}\label{def:RelClass}
\rm{
The pair $(\vd, \vmu)$ as above is called an \emph{effective class} of the FTCY graph $\Gamma$. Let
$$
    \Eff(\Gamma)
$$
denote the set of all effective classes.
}
\end{definition}

The genus-$g$, degree-$(\vd, \vmu)$ formal relative Gromov-Witten invariant
$$
    N^{\hY, \hD}_{g, \vd, \vmu} \in \bQ
$$
is a virtual count of relative stable maps
$$
    u: \left(C, \{q^i_j \mid i = 1, \dots, s, j = 1, \dots, \ell(\mu^i) \} \right) \to (\hY,\hD)
$$
where:
\begin{itemize}
    \item $C$ is a connected prestable (borderless) Riemann surface of arithmetic genus $g$ and $\{q^i_j\}$ is a set of distinct smooth points on $C$.
    
    \item $u_*[C] = \vd \in H_2(\hY;\bZ)$ and $u^{-1}(\hD^i) = \sum_{j=1}^{\ell(\mu^i)} \mu^i_j q^i_j$ as Cartier divisors.
\end{itemize}
The invariant $F^{\hY, \hD}_{g, \vd, \vmu}$ is defined by $T'$-equivariant localization on the corresponding moduli space of relative stable maps; see \cite[Section 4]{LLLZ09}, \cite[Section 3.2]{FL13}. A priori, it is a rational function in the equivariant parameters $\su_1, \su_2$ that is homogeneous of degree zero, but it is in fact a rational number independent of $\su_1, \su_2$ \cite{LLLZ09}. Then, the open Gromov-Witten invariant is defined as
$$
    N^{X, L, \vf}_{g, \beta, \vmu} := (-1)^{|\vmu|-\ell(\vmu)}\sum_{\vd: \pi(\vd) = \beta} N^{\hY, \hD}_{g, \vd, \vmu}.
$$

We consider the generating functions
$$
    F^{X, L, \vf}_{\beta, \vmu}(g_s) := \sum_{g \in \bZ_{\ge 0}} N^{X, L, \vf}_{g, \beta, \vmu} (g_s)^{2g-2+\ell(\vmu)}, \qquad F^{\hY, \hD}_{\vd, \vmu}(g_s) := \sum_{g \in \bZ_{\ge 0}} N^{\hY, \hD}_{g, \vd, \vmu} (g_s)^{2g-2+\ell(\vmu)}
$$
where $g_s$ is a formal variable. Then
\begin{equation}\label{eqn:OpenRelFunction}
    F^{X, L, \vf}_{\beta, \vmu}(g_s) = (-1)^{|\vmu|-\ell(\vmu)}\sum_{\vd: \pi(\vd) = \beta}F^{\hY, \hD}_{\vd, \vmu}(g_s).
\end{equation}
Moreover, let
$$
    \{Q^\beta \mid \beta \in H_2(X,L; \bZ) \text{ effective}\}, \qquad \{Q^{\vd} \mid \vd : E_c(\Gamma) \to \bZ_{\ge 0}\}, \qquad \{P_{\vmu} \mid \vmu \in \cP^s\}
$$
be formal variables with multiplication rules
$$
    Q^{\beta_1}Q^{\beta_2} = Q^{\beta_1 + \beta_2}, \qquad Q^{\vd_1}Q^{\vd_2} = Q^{\vd_1 + \vd_2}, \qquad  P_{\vmu} P_{\vnu} = P_{\vmu \sqcup \vnu},
$$
where for $\vmu = (\mu^1, \dots, \mu^s), \vnu = (\nu^1, \dots, \nu^s) \in \cP^s$ we write $\vmu \sqcup \vnu = (\mu^1 \sqcup \nu^1, \dots, \mu^s \sqcup \nu^s)$. We define
\begin{align*}
    & F^{X, L, \vf}(g_s, Q, P) := \sum_{(\beta, \vmu) \in \Eff(X, L, \vf)} F^{X, L, \vf}_{\beta, \vmu}(g_s)Q^\beta P_{\vmu}, \\
    & F^{\hY, \hD}(g_s, Q, P) := \sum_{(\vd, \vmu) \in \Eff(\Gamma)}F^{\hY, \hD}_{\vd, \vmu}(g_s) Q^{\vd}P_{\vmu}.
\end{align*}

\subsection{The topological vertex}\label{sect:TopVertex}
We now state the computation of $F^{\hY, \hD}(g_s, Q, P)$ using the \emph{topological vertex} formalism \cite{AKMV03, LLLZ09}. We define
$$
    Z^{\hY, \hD}(g_s, Q, P) := \exp \left( F^{\hY, \hD}(g_s, Q, P) \right)
$$
and write
$$
    Z^{\hY, \hD}(g_s, Q, P) = 1 + \sum_{(\vd, \vmu) \in \Eff(\Gamma)}Z^{\hY, \hD}_{\vd, \vmu}(g_s) Q^{\vd}P_{\vmu}.
$$
Geometrically, the coefficient $Z^{\hY, \hD}_{\vd, \vmu}(g_s)$ is the generating function of degree-$(\vd, \vmu)$ formal relative Gromov-Witten invariants of $(\hY, \hD)$ that account for stable maps with possibly disconnected domains (see \cite[Section 4]{LLLZ09}). 

Let $(\vd, \vmu) \in \Eff(\Gamma)$. We set
$$
    T_{\vd} := \{ \vlambda: E^o(\Gamma) \to \cP \mid \vlambda(e) \vdash \vd(\bar{e}), \vlambda(-e) = \vlambda(e)^t\}.
$$
Here $\bar{e} \in E(\Gamma)$ denotes the unoriented edge corresponding to $e \in E^o(\Gamma)$ and we set $\vd(\bar{e}) = 0$ if $\bar{e} \not \in E_c(\Gamma)$. Moreover, for any trivalent vertex $v \in V^3(\Gamma)$, if $e^1, e^2, e^3 \in E^o(\Gamma)$ are the three edges emanating from $v$ and appearing in counterclockwise order in $\Gamma$, we set
$$
    \vlambda^v := (\vlambda(e^1), \vlambda(e^2), \vlambda(e^3))
$$
which is unique up to cyclic symmetry.

\begin{theorem}[\cite{LLLZ09}] \label{thm:TopVertex}
For any $(\vd, \vmu) \in \Eff(\Gamma)$, we have
\begin{equation}\label{eqn:TopVertex}
    Z^{\hY, \hD}_{\vd, \vmu}(q) = \sum_{\vlambda \in T_{\vd}} \prod_{\bar{e} \in E(\Gamma)} (-1)^{(n^e + 1)\vd(\bar{e})} q^{\frac{\kappa_{\vlambda(e)}n^e}{2}} \prod_{v \in V^3(\Gamma)} \cW_{\vlambda^v}(q) \prod_{i = 1}^s \frac{\chi_{\vlambda(-e_i)}(\mu^i)}{z_{\mu^i}}\sqrt{-1}^{\ell(\mu^i)}(-1)^{|\mu^i|}
\end{equation}
under the change of variables
$$
    q = e^{\sqrt{-1}g_s}.
$$
\end{theorem}

Here and throughout the paper, we fix a choice of $\sqrt{-1}$. We further explain the notations in \eqref{eqn:TopVertex}. For each $\bar{e} \in E(\Gamma)$, $e \in E^o(\Gamma)$ is a choice of an orientation of $\bar{e}$ and the choice does not affect \eqref{eqn:TopVertex} since
$$
    n^{-e} = -n^e, \qquad \kappa_{\vlambda(-e)} = \kappa_{\vlambda(e)^t} = -\kappa_{\vlambda(e)}.
$$
For each $v \in V^3(\Gamma)$, $\cW_{\vlambda^v}(q)$ is the function defined in Definition \ref{def:ThreePoint} associated to the triple $\vlambda^v \in \cP^3$. Finally, for any $\lambda, \mu \in \cP$ with $|\lambda| = |\mu| = d$, $\chi_\lambda$ denotes the irreducible character of $S_d$ indexed by $\lambda$ and $\chi_\lambda(\mu) \in \bZ$ is the value on the conjugacy class indexed by $\mu$. Recall that $e_i$ is the edge $\bar{e}_i$ oriented to terminate at the univalent vertex $v_i \in V_1(\Gamma)$.

Theorem \ref{thm:TopVertex} is obtained from \cite[Proposition 7.4, Corollary 7.6]{LLLZ09} by using the combinatorial expression $\cW_{\vlambda^v}$ to substitute the expression $\tC_{\vlambda^v}$ defined in \cite[Section 6.4]{LLLZ09} via Gromov-Witten invariants. The substitution is valid by results in \cite[Section 8]{LLLZ09} and \cite{MOOP11}. We note that there are sign discrepancies between \eqref{eqn:TopVertex} and \cite[Proposition 7.4]{LLLZ09}, which we address in detail in Appendix \ref{appdx:TopVertex}.

%% file: integrality.tex
\section{Open BPS invariants, integrality, and finiteness}\label{sect:OpenBPS}
In this section, we define the open BPS invariants of $(X,L,\vf)$ via resummation of the open Gromov-Witten invariants and prove their integrality and finiteness properties (Theorem \ref{thm:OpenBPS}).

\subsection{Definitions and statements}
Given effective class $(\beta, \vmu) \in \Eff(X,L, \vf)$, we consider the following resummation formula of Labastida-Mari\~no-Ooguri-Vafa (LMOV) \cite{OV00,LM00,LMV00,MV02}:
\begin{equation}\label{eqn:LMOVResumGs}
    \begin{aligned}
        F^{X, L, \vf}_{\beta, \vmu}(g_s) & = \sum_{g \in \bZ_{\ge 0}} N^{X, L, \vf}_{g, \beta, \vmu} g_s^{2g-2+\ell(\vmu)}\\
        &= \frac{1}{\prod_{i = 1}^s z_{\mu^i}} \sum_{k \mid \beta, \vmu}  \sum_{g \in \bZ_{\ge 0}} (-1)^{\ell(\vmu)+g}k^{\ell(\vmu)-1}n^{X, L, \vf}_{g,\frac{\beta}{k}, \frac{\vmu}{k}}\left(2\sin \frac{kg_s}{2} \right)^{2g-2} \prod_{i=1}^s \prod_{j=1}^{\ell(\mu^i)} 2\sin \frac{\mu^i_jg_s}{2}\\
        &= \sum_{k \mid \beta, \vmu} \sum_{g \in \bZ_{\ge 0}} \frac{(-1)^{\ell(\vmu)+g}n^{X, L, \vf}_{g,\frac{\beta}{k}, \frac{\vmu}{k}}}{k \prod_{i = 1}^s z_{\frac{\mu^i}{k}}} \left(2\sin \frac{kg_s}{2} \right)^{2g-2} \prod_{i=1}^s \prod_{j=1}^{\ell(\mu^i)} 2\sin \frac{\mu^i_jg_s}{2}.
    \end{aligned}
\end{equation}
Here for $k \in \bZ_{\ge 1}$, we say that $k \mid \beta$ if $\frac{\beta}{k} \in H_2(X,L;\bZ)$, and $k \mid \vmu$ if $k \mid \mu^i$ for each $i$ (or equivalently $\frac{\vmu}{k}$ exists). We refer to the coefficient
$$
    n^{X, L, \vf}_{g, \beta, \vmu} \in \bQ
$$
determined by \eqref{eqn:LMOVResumGs} as the genus-$g$, degree-$(\beta, \vmu)$ \emph{open BPS invariant} of $(X, L, \vf)$. Our main results of this section are the following properties of these invariants.

\begin{theorem}\label{thm:OpenBPS}
For any $(\beta, \vmu) \in \Eff(X, L, \vf)$, we have that
$$
    n^{X, L, \vf}_{g, \beta, \vmu} \in \bZ
$$
for all $g \in \bZ_{\ge 0}$ and is zero for $g \gg 0$.
\end{theorem}

To prove Theorem \ref{thm:OpenBPS}, we first rewrite the resummation formula \eqref{eqn:LMOVResumGs} using the change of variables $q = e^{\sqrt{-1}g_s}$. For $a \in \bZ_{\ge 1}$, we write
$$
    [a]:= q^{\frac{a}{2}} - q^{-\frac{a}{2}},
$$
and for any partition $\lambda \in \cP$, we write
$$
    [\lambda]:= \prod_{j=1}^{\ell(\lambda)} [\lambda_j].
$$
Note that
$$
    2\sin \frac{ag_s}{2} = \frac{[a]}{\sqrt{-1}}.
$$
Then \eqref{eqn:LMOVResumGs} gives
\begin{equation}\label{eqn:LMOVResumQ}
    F^{X, L, \vf}_{\beta, \vmu}(q) = \sum_{k \mid \beta, \vmu} \sum_{g \in \bZ_{\ge 0}} \frac{-\sqrt{-1}^{\ell(\frac{\vmu}{k})}n^{X, L, \vf}_{g,\frac{\beta}{k}, \frac{\vmu}{k}}}{k \prod_{i = 1}^s z_{\frac{\mu^i}{k}}} [k]^{2g-2} \prod_{i=1}^s \prod_{j=1}^{\ell(\frac{\mu^i}{k})} [k \cdot \frac{\mu^i_j}{k}].
\end{equation}
We define an auxiliary function
$$
    H^{X, L, \vf}_{\beta, \vmu}(q) := \sum_{g \in \bZ_{\ge 0}} \frac{-\sqrt{-1}^{\ell(\vmu)}n^{X, L, \vf}_{g, \beta, \vmu}}{ \prod_{i = 1}^s z_{\mu^i}} [1]^{2g-2} \prod_{i=1}^s [\mu^i].
$$
Then \eqref{eqn:LMOVResumQ} can be rewritten as
$$
    F^{X, L, \vf}_{\beta, \vmu}(q) = \sum_{k \mid \beta, \vmu} \frac{1}{k} H^{X, L, \vf}_{\frac{\beta}{k}, \frac{\vmu}{k}}(q^k)
$$
which is equivalent to
\begin{equation}\label{eqn:HtoF}
    H^{X, L, \vf}_{\beta, \vmu}(q) = \sum_{k \mid \beta, \vmu} \frac{\mu(k)}{k} F^{X, L, \vf}_{\frac{\beta}{k}, \frac{\vmu}{k}}(q^k).
\end{equation}
Here, $\mu(k)$ is the \emph{M\"obius function} and satisfies that
$$
    \sum_{k' \mid k} \mu(k') = \begin{cases}
        1 & \text{if } k=1,\\
        0 & \text{for } k \in \bZ_{\ge 2}.
    \end{cases}
$$

Moreover, we set up a new formal variable
$$
    t := [1]^2 = q + q^{-1} -2
$$
and define a generating function
\begin{equation}\label{eqn:GDef}
    G^{X, L, \vf}_{\beta, \vmu}(t) := \sum_{g \in \bZ_{\ge 0}} -n^{X, L, \vf}_{g, \beta, \vmu} t^{g-1}.
\end{equation}
The definitions imply that
\begin{equation}\label{eqn:GtoH}
    G^{X, L, \vf}_{\beta, \vmu}(t) = H^{X, L, \vf}_{\beta, \vmu}(q) \prod_{i=1}^s \frac{z_{\mu^i}}{\sqrt{-1}^{\ell(\mu^i)}[\mu^i]}. 
\end{equation}

The proof of Theorem \ref{thm:OpenBPS} is based on the following two technical lemmas, which are the counterparts of \cite[Propositions 3.2, 3.3]{Konishi06b} respectively.

\begin{lemma}\label{lem:GIntegral}
For any $(\beta, \vmu) \in \Eff(X, L, \vf)$, we have that
$$
    G^{X, L, \vf}_{\beta, \vmu}(t) \in \cL[t].
$$
\end{lemma}

Here,
$$
    \cL[t] := \left\{ \frac{a(t)}{b(t)} \biggm| a(t) \in \bZ[t], b(t) \in \bZ_0[t], b(t) \neq 0  \right\}
$$
where $\bZ_0[t] \subset \bZ[t]$ is the subring of monic polynomials in $t$.

\begin{lemma}\label{lem:GFinite}
For any $(\beta, \vmu) \in \Eff(X, L, \vf)$, we have that
$$
    tG^{X, L, \vf}_{\beta, \vmu}(t) \in \bQ[t]. 
$$
\end{lemma}

The proofs of the lemmas will be given in Sections \ref{sect:ProofIntegral}, \ref{sect:ProofFinite} below respectively.

\begin{proof}[Proof of Theorem \ref{thm:OpenBPS}]
As observed by \cite{Konishi06a,Konishi06b}, Lemmas \ref{lem:GIntegral} and \ref{lem:GFinite} together imply that
$$
    tG^{X, L, \vf}_{\beta, \vmu}(t) \in \bZ[t]
$$
which is equivalent to the theorem.
\end{proof}

\subsection{Proof of Lemma \ref{lem:GIntegral}}\label{sect:ProofIntegral}
Now we prove Lemma \ref{lem:GIntegral}. We start by using \eqref{eqn:GtoH}, \eqref{eqn:HtoF}, \eqref{eqn:OpenRelFunction} to rewrite the generating function $G^{X, L, \vf}_{\beta, \vmu}(t)$ as follows:
\begin{equation}\label{eqn:GRewrite1}
\begin{aligned}
    G^{X, L, \vf}_{\beta, \vmu}(t) & = H^{X, L, \vf}_{\beta, \vmu}(q) \prod_{i=1}^s \frac{z_{\mu^i}}{\sqrt{-1}^{\ell(\mu^i)}[\mu^i]}\\
    & = \prod_{i=1}^s \frac{z_{\mu^i}}{\sqrt{-1}^{\ell(\mu^i)}[\mu^i]}  \sum_{k \mid \beta, \vmu} \frac{\mu(k)}{k} F^{X, L, \vf}_{\frac{\beta}{k}, \frac{\vmu}{k}}(q^k)\\
    & = \prod_{i=1}^s \frac{\sqrt{-1}^{\ell(\mu^i)}z_{\mu^i}}{[\mu^i]} \sum_{k \mid \beta, \vmu} \frac{\mu(k)}{k} (-1)^{|\frac{\vmu}{k}|}\sum_{\vd: \pi(\vd) = \frac{\beta}{k}} F^{\hY, \hD}_{g, \vd, \frac{\vmu}{k}}(q^k).
\end{aligned}
\end{equation}

Now we bring in the topological vertex to continue rewriting \eqref{eqn:GRewrite1}, using techniques in \cite[Section 5]{Konishi06b}. Note that $F^{\hY, \hD}_{g, \vd, \frac{\vmu}{k}}(q^k)$ is the coefficient of $Q^{\vd}P_{\frac{\vmu}{k}}$ in
$$
    \ln \left(Z^{\hY, \hD}(q^k, Q, P)\right).
$$
To extract this coefficient, we introduce some notations. For any effective class $(\vd, \vmu) \in \Eff(\Gamma)$, we denote
$$
    \cD(\vd, \vmu) := \{ (\vdelta, \vnu) \in \Eff(\Gamma) \mid \vdelta \le \vd, \vnu \subseteq \vmu\},
$$
and
$$
    \cA(\vd, \vmu) := \left\{ \va: \cD(\vd, \vmu) \to \bZ_{\ge 0} \biggm| \sum_{(\vdelta, \vnu) \in \cD(\vd, \vmu)} \va(\vdelta, \vnu) \vdelta = \vd, \bigsqcup_{(\vdelta, \vnu) \in \cD(\vd, \vmu)} \vnu^{(\va(\vdelta, \vnu))} = \vmu \ \right\}
$$
which is the set of $\va$ such that
$$
    \prod_{(\vdelta, \vnu) \in \cD(\vd, \vmu)} (Q^{\vdelta}P_{\vnu})^{\va(\vdelta, \vnu)} = Q^{\vd}P_{\vmu}.
$$
For $\va \in \cA(\vd, \vmu)$, we write
$$
    |\va| := \sum_{(\vdelta, \vnu) \in \cD(\vd, \vmu)} \va(\vdelta, \vnu), \qquad \gcd(\va) := \gcd(\{\va(\vdelta, \vnu) \mid (\vdelta, \vnu) \in \cD(\vd, \vmu) \}).
$$
Moreover, we denote
$$
    \cA_1(\vd, \vmu) := \{\va \in \cA(\vd, \vmu) \mid \gcd(\va)=1\}.
$$

Returning to \eqref{eqn:GRewrite1}, for any $\vd$ such that $\pi(\vd) = \frac{\beta}{k}$, we have
\begin{equation}\label{eqn:FRewrite}
\begin{aligned}
    F^{\hY, \hD}_{g, \vd, \frac{\vmu}{k}}(q^k) = & \sum_{\va \in \cA(\vd, \frac{\vmu}{k})} \frac{|\va|!}{\prod_{(\vdelta, \vnu) \in \cD(\vd, \frac{\vmu}{k})} \va(\vdelta, \vnu)!} \frac{(-1)^{|\va|-1}}{|\va|} \prod_{(\vdelta, \vnu) \in \cD(\vd, \frac{\vmu}{k})} \left(Z^{\hY, \hD}_{\vdelta, \vnu}(q^k)\right)^{\va(\vdelta, \vnu)}\\
    = & \sum_{\substack{k' \mid \vd \\ (\frac{\vmu}{k})^{(\frac{1}{k'})} \text{ exists}}} \sum_{\va \in \cA_1\left(\frac{\vd}{k'}, (\frac{\vmu}{k})^{(\frac{1}{k'})}\right)} \frac{(k'|\va|)!}{\displaystyle\prod_{(\vdelta, \vnu) \in \cD\left(\frac{\vd}{k'}, (\frac{\vmu}{k})^{(\frac{1}{k'})}\right)} (k'\va(\vdelta, \vnu))!} \frac{(-1)^{k'|\va|-1}}{k'|\va|} \\
    &  \cdot \prod_{(\vdelta, \vnu) \in \cD\left(\frac{\vd}{k'}, (\frac{\vmu}{k})^{(\frac{1}{k'})}\right)} \left(Z^{\hY, \hD}_{\vdelta, \vnu}(q^k)\right)^{k'\va(\vdelta, \vnu)}.
\end{aligned}
\end{equation}
By Theorem \ref{thm:TopVertex}, for $(\vdelta, \vnu) \in \cD\left(\frac{\vd}{k'}, (\frac{\vmu}{k})^{(\frac{1}{k'})}\right)$ as above,
$$
    Z^{\hY, \hD}_{\vdelta, \vnu}(q^k) = \sum_{\vlambda \in T_{\vdelta}} \prod_{\bar{e} \in E(\Gamma)} (-1)^{(n^e + 1)\vdelta(\bar{e})} q^{\frac{k\kappa_{\vlambda(e)}n^e}{2}} \prod_{v \in V^3(\Gamma)} \cW_{\vlambda^v}(q^k) \prod_{i = 1}^s \frac{\chi_{\vlambda(-e_i)}(\nu^i)}{z_{\nu^i}}\sqrt{-1}^{\ell(\nu^i)}(-1)^{|\nu^i|}.
$$
This combined with \eqref{eqn:GRewrite1} gives
\begin{equation}\label{eqn:GRewrite2}
    \begin{aligned}
        & G^{X, L, \vf}_{\beta, \vmu}(t) = \prod_{i=1}^s \frac{\sqrt{-1}^{\ell(\mu^i)}z_{\mu^i}}{[\mu^i]} \sum_{k \mid \beta, \vmu} \frac{\mu(k)}{k} (-1)^{|\frac{\vmu}{k}|}  \sum_{\vd: \pi(\vd) = \frac{\beta}{k}} \sum_{\substack{k' \mid \vd \\ (\frac{\vmu}{k})^{(\frac{1}{k'})} \text{ exists}}}\\
        \phantom{aaaa} &  \sum_{\va \in \cA_1\left(\frac{\vd}{k'}, (\frac{\vmu}{k})^{(\frac{1}{k'})}\right)} \frac{(k'|\va|)!}{\displaystyle\prod_{(\vdelta, \vnu) \in \cD\left(\frac{\vd}{k'}, (\frac{\vmu}{k})^{(\frac{1}{k'})}\right)} (k'\va(\vdelta, \vnu))!} \frac{(-1)^{k'|\va|-1}}{k'|\va|}  \prod_{(\vdelta, \vnu) \in \cD\left(\frac{\vd}{k'}, (\frac{\vmu}{k})^{(\frac{1}{k'})}\right)}\\
        \phantom{aaaa} & \left(\sum_{\vlambda \in T_{\vdelta}} \prod_{\bar{e} \in E(\Gamma)} (-1)^{(n^e + 1)\vdelta(\bar{e})} q^{\frac{k\kappa_{\vlambda(e)}n^e}{2}} \prod_{v \in V^3(\Gamma)} \cW_{\vlambda^v}(q^k) \prod_{i = 1}^s \frac{\chi_{\vlambda(-e_i)}(\nu^i)}{z_{\nu^i}}\sqrt{-1}^{\ell(\nu^i)}(-1)^{|\nu^i|}\right)^{k'\va(\vdelta, \vnu)}.
    \end{aligned}
\end{equation}
Here, we have by definition that
$$
    \frac{\vmu}{k} = \bigsqcup_{(\vdelta, \vnu) \in \cD\left(\frac{\vd}{k'}, (\frac{\vmu}{k})^{(\frac{1}{k'})}\right)} \vnu^{(k'\va(\vdelta, \vnu))}
$$
and in particular
$$
    \ell(\vmu) = \ell\left(\frac{\vmu}{k}\right) = \sum_{(\vdelta, \vnu) \in \cD\left(\frac{\vd}{k'}, (\frac{\vmu}{k})^{(\frac{1}{k'})}\right)} k'\va(\vdelta, \vnu) \ell(\vnu), \qquad 
    \left|\frac{\vmu}{k}\right| =  \sum_{(\vdelta, \vnu) \in \cD\left(\frac{\vd}{k'}, (\frac{\vmu}{k})^{(\frac{1}{k'})}\right)} k'\va(\vdelta, \vnu) |\vnu|.
$$
Then \eqref{eqn:GRewrite2} simplifies to
\begin{equation}\label{eqn:GRewrite3}
    \begin{aligned}
        G^{X, L, \vf}_{\beta, \vmu}(t) = & (-1)^{\ell(\vmu)}\prod_{i=1}^s z_{\mu^i} \sum_{k \mid \beta, \vmu} \frac{\mu(k)}{k} \sum_{\vd: \pi(\vd) = \frac{\beta}{k}} \sum_{\substack{k' \mid \vd \\ (\frac{\vmu}{k})^{(\frac{1}{k'})} \text{ exists}}} \\
        &  \sum_{\va \in \cA_1\left(\frac{\vd}{k'}, (\frac{\vmu}{k})^{(\frac{1}{k'})}\right)} \frac{(k'|\va|)!}{\displaystyle\prod_{(\vdelta, \vnu) \in \cD\left(\frac{\vd}{k'}, (\frac{\vmu}{k})^{(\frac{1}{k'})}\right)} (k'\va(\vdelta, \vnu))!} \frac{(-1)^{k'|\va|-1}}{k'|\va|}  \prod_{(\vdelta, \vnu) \in \cD\left(\frac{\vd}{k'}, (\frac{\vmu}{k})^{(\frac{1}{k'})}\right)}\\
        & \left(\sum_{\vlambda \in T_{\vdelta}} \prod_{\bar{e} \in E(\Gamma)} (-1)^{(n^e + 1)\vdelta(\bar{e})} q^{\frac{k\kappa_{\vlambda(e)}n^e}{2}} \prod_{v \in V^3(\Gamma)} \cW_{\vlambda^v}(q^k) \prod_{i = 1}^s \frac{\chi_{\vlambda(-e_i)}(\nu^i)}{z_{\nu^i} [k\nu^i]}\right)^{k'\va(\vdelta, \vnu)}.
    \end{aligned}
\end{equation}
Moreover, we have
$$
    k^{\ell(\vmu)} (k'!)^s \biggm|  \frac{\displaystyle\prod_{i=1}^s z_{\mu^i}}{\displaystyle \prod_{(\vdelta, \vnu) \in \cD\left(\frac{\vd}{k'}, (\frac{\vmu}{k})^{(\frac{1}{k'})}\right)} \prod_{i=1}^s z_{\nu^i}^{k'\va(\vdelta, \vnu)}}
$$
where the factor $(k'!)^s$ comes from permutations the $k'$ copies of the $\nu^i$'s in the automorphism group, which implies that
$$
    \frac{1}{kk'} \frac{\displaystyle\prod_{i=1}^s z_{\mu^i}}{\displaystyle \prod_{(\vdelta, \vnu) \in \cD\left(\frac{\vd}{k'}, (\frac{\vmu}{k})^{(\frac{1}{k'})}\right)} \prod_{i=1}^s z_{\nu^i}^{k'\va(\vdelta, \vnu)}} \in \bZ
$$
since $\ell(\vmu), s \ge 1$. By \cite[Lemma A.2]{Konishi06b}, we also have
$$
    \frac{(k'|\va|)!}{\displaystyle |\va| \prod_{(\vdelta, \vnu) \in \cD\left(\frac{\vd}{k'}, (\frac{\vmu}{k})^{(\frac{1}{k'})}\right)} (k'\va(\vdelta, \vnu))!} \in \bZ.
$$
We may therefore rewrite \eqref{eqn:GRewrite3} as
\begin{align*}
    G^{X, L, \vf}_{\beta, \vmu}(t) = & \sum_{k \mid \beta, \vmu} \sum_{\vd: \pi(\vd) = \frac{\beta}{k}} \sum_{\substack{k' \mid \vd \\ (\frac{\vmu}{k})^{(\frac{1}{k'})} \text{ exists}}} \sum_{\va \in \cA_1\left(\frac{\vd}{k'}, (\frac{\vmu}{k})^{(\frac{1}{k'})}\right)}  \prod_{(\vdelta, \vnu) \in \cD\left(\frac{\vd}{k'}, (\frac{\vmu}{k})^{(\frac{1}{k'})}\right)} c_{k,\vd,k',\va,\vnu}\\
    & \left(\sum_{\vlambda \in T_{\vdelta}} \prod_{\bar{e} \in E(\Gamma)} (-1)^{(n^e + 1)\vdelta(\bar{e})} q^{\frac{k\kappa_{\vlambda(e)}n^e}{2}} \prod_{v \in V^3(\Gamma)} \cW_{\vlambda^v}(q^k) \prod_{i = 1}^s \frac{\chi_{\vlambda(-e_i)}(\nu^i)}{[k\nu^i]}\right)^{k'\va(\vdelta, \vnu)}.
\end{align*}
for some $c_{k,\vd,k',\va,\vnu} \in \bZ$.

To prove Lemma \ref{lem:GIntegral}, it suffices to show that for any $k,\vd,k',\va, (\vdelta,\vnu)$ as above,
$$
    \prod_{i=1}^s \frac{1}{[k\nu^i]^2} \sum_{\vlambda \in T_{\vdelta}} Y_{\vlambda}(q) \in \cL[t],
$$
where for $\vlambda \in T_{\vdelta}$ we set
$$
    Y_{\vlambda}(q) := \prod_{\bar{e} \in E(\Gamma)} (-1)^{(n^e + 1)\vdelta(\bar{e})} q^{\frac{k\kappa_{\vlambda(e)}n^e}{2}} \prod_{v \in V^3(\Gamma)} \cW_{\vlambda^v}(q^k) \prod_{i = 1}^s \chi_{\vlambda(-e_i)}(\nu^i)[k\nu^i].
$$
By \cite[Lemma P1]{BP01}, $[a]^2 \in \bZ_0[t]$ for any $a \in \bZ_{\ge 1}$, and thus it is further reduced to showing that
\begin{equation}\label{eqn:SumYq}
    \sum_{\vlambda \in T_{\vdelta}} Y_{\vlambda}(q) \in \cL[t].
\end{equation}
By \cite[Lemma 5.3(vii)]{Konishi06b},
$$
    \cW_{\vlambda^v}(q^k) \in q^{\frac{k|\vlambda^v|}{2}}\frac{\bZ[q,q^{-1}]}{\bZ_0[t]}.
$$
Note that for each edge $e_i$, since $|\vlambda(e_i)| = |\nu^i|$, we have
$$
    q^{\frac{k|\vlambda(e_i)|}{2}}[k\nu^i] \in \bZ[q].
$$
Therefore,
$$
    Y_{\vlambda}(q) \in \frac{\bZ[q,q^{-1}]}{\bZ_0[t]}.
$$
Now for $\vlambda \in T_{\vdelta}$, let $\vlambda^t \in T_{\vdelta}$ denote the element such that $\vlambda^t(e) = (\vlambda(e))^t$. By \cite[Lemma 5.3(viii)]{Konishi06b},
$$
    \cW_{(\vlambda^t)^v}(q^k) = (-1)^{|\vlambda^v|}\cW_{\vlambda^v}(q^{-k}).
$$
For each edge $e_i$, we have
$$
    (-1)^{|\vlambda^t(e_i)|}\chi_{\vlambda^t(-e_i)}(\nu^i)[k\nu^i] = \chi_{\vlambda(-e_i)}(\nu^i)[k\nu^i] \bigg|_{q \to q^{-1}}.
$$
Therefore,
$$
    Y_{\vlambda^t}(q) = Y_{\vlambda}(q^{-1}).
$$
For $\vlambda \in T_{\vdelta}$ such that $\vlambda^t = \vlambda$, $Y_{\vlambda}(q)$ is an element in $\frac{\bZ[q,q^{-1}]}{\bZ_0[t]}$ that is symmetric in $q$ and $q^{-1}$, and is thus in $\cL[t]$ by \cite[Lemma 6.2]{Konishi06a}. If otherwise $\vlambda^t \neq \vlambda$, similarly we have $Y_{\vlambda}(q) + Y_{\vlambda^t}(q) \in \cL[t]$. Summarizing the two cases gives \eqref{eqn:SumYq} and completes the proof. \qed

\subsection{An auxiliary FTCY graph}\label{sect:AuxGraph}
In this section, we prepare for the proof of Lemma \ref{lem:GFinite} by collecting additional notations and results. The main idea of the proof is to relate $G^{X, L, \vf}_{\beta, \vmu}(t)$ to generating functions defined by a trivalent FTCY graph constructed from $\Gamma$ and apply the analysis of this auxiliary graph given by \cite[Section 6]{Konishi06b}. Let $\Gamma'$ be the FTCY graph obtained from $\Gamma$ by adding two directed edges to each univalent vertex $v_i$ to make it a trivalent vertex. The position vectors of the new edges are chosen such that the degree $n^{e_i}$ is preserved for all $i = 1, \dots, s$. See Figure \ref{fig:AuxGraph} for an illustration. We have
$$
    V(\Gamma') = V^3(\Gamma') = V(\Gamma), \qquad E_c(\Gamma') = E_c^3(\Gamma') = E_c(\Gamma).
$$

\begin{figure}[h]
    \begin{tikzpicture}[scale=0.7]
        \coordinate (v0) at (0, 0);
        \coordinate (v1) at (-2, 0);        
                
        \node at (v0) {$\bullet$};
        \node at (v1) {$\bullet$};

        \draw (-2,0.8) -- (v1) -- (-2.8,-0.8);
        \draw (v1) -- (v0);
        \draw[->] (v1) -- (-0.95, 0);
        \draw[<->, dashed] (0,1) -- (0, -1);        

        \node at (0.5, 0) {$v_i$};
        \node[above] at (-1, 0) {$e_i$};
        \node[above] at (-2, 1) {$\cdots$};
        \node[below] at (-3, -1) {$\cdots$};
        \node at (-1, -2) {$\Gamma$};

        \node at (4, 0) {$\Rightarrow$};

        \coordinate (v10) at (10, 0);
        \coordinate (v11) at (8, 0);        
                
        \node at (v10) {$\bullet$};
        \node at (v11) {$\bullet$};

        \draw (8,0.8) -- (v11) -- (7.2,-0.8);
        \draw (v11) -- (v10);
        \draw[->] (v11) -- (9.05, 0);
        \draw[->] (v10) -- (10, -1);
        \draw[->] (v10) -- (11, 1);        

        \node at (10.5, 0) {$v_i$};
        \node[above] at (9, 0) {$e_i$};
        \node[above] at (8, 1) {$\cdots$};
        \node[below] at (7, -1) {$\cdots$};
        \node at (9, -2) {$\Gamma'$};

    \end{tikzpicture}

    \caption{Construction of $\Gamma'$ at vertex $v_i$.}
    \label{fig:AuxGraph}
\end{figure}

Let $\hY'$ be the FTCY 3-fold defined by $\Gamma'$. Then any effective class $\vd$ of $\hY$ can also be viewed as an effective class of $\Gamma'$ or $\hY'$, and we set
$$
    \Eff(\Gamma') := \{\vd: E_c(\Gamma') \to \bZ_{\ge 0} \mid \vd \neq 0\}.
$$
Following Theorem \ref{thm:TopVertex}, we define
$$ 
    Z^{\hY'}_{\vd}(q) = \sum_{\vlambda \in T_{\vd}} \prod_{\bar{e} \in E(\Gamma')} (-1)^{(n^e + 1)\vd(\bar{e})} q^{\frac{\kappa_{\vlambda(e)}n^e}{2}} \prod_{v \in V^3(\Gamma')} \cW_{\vlambda^v}(q).
$$
for each $\vd \in \Eff(\Gamma')$ and
$$
    Z'(q, Q) := 1 + \sum_{\vd \in \Eff(\Gamma')}Z^{\hY'}_{\vd}(q) Q^{\vd}.
$$
Moreover, we set
$$
    F^{\hY'}(q, Q): = \ln \left( Z^{\hY'}(q, Q)\right) = \sum_{\vd \in \Eff(\Gamma')}F^{\hY'}_{\vd}(q) Q^{\vd}.
$$
Then $Z^{\hY'}$ and $F^{\hY'}$ can be interpreted as generating functions of formal \emph{closed} Gromov-Witten invariants of $\hY'$.

The analysis of \cite[Section 6]{Konishi06b} expresses $Z^{\hY'}$ and $F^{\hY'}$ as sums of contributions indexed by certain sets of labeled graphs. We now briefly summarize the relevant results. Let $F(\Gamma')$ denote the set of flags in $\Gamma'$ and $F^3(\Gamma')$ denote the subset of flags $f = (v,e)$ such that $\bar{e} \in E_c(\Gamma')$. For each vertex $v \in V(\Gamma')$, we fix a counterclockwise ordering of the three flags at $v$ by $1, 2, 3$. We $f_1(v), f_2(v), f_3(v) \in F(\Gamma')$ to denote the flag at a vertex $v \in V(\Gamma')$ ordered by $1, 2, 3$ respectively, and $\iota(f) \in \{1, 2, 3\}$ to denote the order of a flag $f \in F(\Gamma')$. Moreover, we require that the ordering at each $v_i$ is chosen such that $\iota(v_i, -e_i) = 2$. 

Let $\vd \in \Eff(\Gamma')$. Let
$$
    D_{\vd}
$$
denote the set of \emph{$\Gamma'$-sets of degree $\vd$}, which are pairs
$$
    \left(\vnu_V = (\nu^v)_{v \in V(\Gamma')}, \vnu_F = (\nu^f)_{f \in F^3(\Gamma')} \right)
$$
of tuples of partitions satisfying the following conditions:
\begin{itemize}
    \item For $f = (v,e) \in F^3(\Gamma')$, $|\nu^v| + |\nu^f| = \vd(\bar{e})$ if $\iota(f) \in \{1,3\}$.
    \item For $f = (v,e) \in F^3(\Gamma')$, $|\nu^f| = \vd(\bar{e})$ if $\iota(f) = 2$.
    \item For $v \in V(\Gamma')$, $\nu^v = \emptyset$ if $f_1(v) \not \in F^3(\Gamma')$ or $f_3(v) \not \in F^3(\Gamma')$.
\end{itemize}
In particular, by our condition at the vertices $v_i$'s, $\nu^{v_i} = \emptyset$ and $|\nu^{(v_i, -e_i)}| = \vd(\bar{e}_i)$ for each $i = 1, \dots, s$. For any $\vmu = (\mu^1, \dots, \mu^s) \in \cP^s$ such that $|\mu^i| = \vd(\bar{e}_i)$ for each $i$, we set
$$
    D_{\vd, \vmu} := \{\left(\vnu_V , \vnu_F  \right) \in D_{\vd} \mid \nu^{(v_i, -e_i)} = \mu^i \text{ for } i = 1, \dots, s\}.
$$

Given any $\vlambda \in T_{\vd}$, at each vertex $v \in V(\Gamma')$, the three point function $\cW_{\vlambda^v}$ can be written as a sum of contributions indexed by the choice of partitions $\nu^v, \nu^{f_1(v)}, \nu^{f_2(v)}, \nu^{f_3(v)}$ (\cite[Lemma 6.2]{Konishi06b}). As a result, \cite{Konishi06b} constructed a set of (possibly disconnected) labeled graphs $\Comb_{\Gamma'}^\bullet(\vnu_V, \vnu_F)$ for each $(\vnu_V, \vnu_F) \in D_{\vd}$ and assigned an amplitude $\cH(W)(q)$ for each labeled graph $W \in \Comb_{\Gamma'}^\bullet(\vnu_V, \vnu_F)$ such that
\begin{equation}\label{eqn:KonishiZ}
    Z^{\hY'}_{\vd}(q) = \sum_{(\vnu_V, \vnu_F) \in D_{\vd}} \frac{1}{z_{\vnu_V}z_{\vnu_F}} \sum_{W \in \Comb_{\Gamma'}^\bullet(\vnu_V, \vnu_F)} \cH(W)(q)
\end{equation}
(\cite[Proposition 6.11]{Konishi06b}). Moreover, if $\Comb_{\Gamma'}^\circ(\vnu_V, \vnu_F) \subseteq \Comb_{\Gamma'}^\bullet(\vnu_V, \vnu_F)$ denotes the subset of connected graphs, then
\begin{equation}\label{eqn:KonishiF}
    F^{\hY'}_{\vd}(q) = \sum_{(\vnu_V, \vnu_F) \in D_{\vd}} \frac{1}{z_{\vnu_V}z_{\vnu_F}} \sum_{W \in \Comb_{\Gamma'}^\circ(\vnu_V, \vnu_F)} \cH(W)(q)
\end{equation}
(\cite[Proposition 6.13]{Konishi06b}).

\subsection{Proof of Lemma \ref{lem:GFinite}}\label{sect:ProofFinite}
We start by considering the following expression of $G^{X, L, \vf}_{\beta, \vmu}(t)$ which follows from the first line of \eqref{eqn:FRewrite} in the same way as how \eqref{eqn:GRewrite3} is derived:
\begin{equation}\label{eqn:GRewrite4}
    \begin{aligned}
        G^{X, L, \vf}_{\beta, \vmu}(t) = & (-1)^{\ell(\vmu)}z_{\vmu} \sum_{k \mid \beta, \vmu} \frac{\mu(k)}{k} \sum_{\vd: \pi(\vd) = \frac{\beta}{k}} \sum_{\va \in \cA\left(\vd, \frac{\vmu}{k}\right)} \frac{|\va|!}{\displaystyle\prod_{(\vdelta, \vnu) \in \cD\left(\vd, \frac{\vmu}{k}\right)} \va(\vdelta, \vnu)!} \frac{(-1)^{|\va|-1}}{|\va|}  \\
        & \prod_{(\vdelta, \vnu) \in \cD\left(\vd, \frac{\vmu}{k}\right)} \left(\sum_{\vlambda \in T_{\vdelta}} \prod_{\bar{e} \in E(\Gamma)} (-1)^{(n^e + 1)\vdelta(\bar{e})} q^{\frac{k\kappa_{\vlambda(e)}n^e}{2}} \prod_{v \in V^3(\Gamma)} \cW_{\vlambda^v}(q^k) \prod_{i = 1}^s \frac{\chi_{\vlambda(-e_i)}(\nu^i)}{z_{\nu^i} [k\nu^i]}\right)^{\va(\vdelta, \vnu)}.
    \end{aligned}
\end{equation}
Motivated by this expression, we define
$$
    Z'_{\vd,\vmu}(q) := \sum_{\vlambda \in T_{\vd}} \prod_{\bar{e} \in E(\Gamma)} (-1)^{(n^e + 1)\vd(\bar{e})} q^{\frac{\kappa_{\vlambda(e)}n^e}{2}} \prod_{v \in V^3(\Gamma)} \cW_{\vlambda^v}(q) \prod_{i = 1}^s \frac{\chi_{\vlambda(-e_i)}(\mu^i)}{z_{\mu^i} [\mu^i]}
$$
for each $(\vd, \vmu) \in \Eff(\Gamma)$ and
$$
    Z'(q, Q, P) := 1 + \sum_{(\vd, \vmu) \in \Eff(\Gamma)}Z'_{\vd, \vmu}(q) Q^{\vd}P_{\vmu}.
$$
Moreover, we set
$$
    F'(q, Q, P): = \ln \left( Z'(q, Q, P)\right) = \sum_{(\vd, \vmu) \in \Eff(\Gamma)}F'_{\vd, \vmu}(q) Q^{\vd}P_{\vmu}.
$$
Then \eqref{eqn:GRewrite4} can be rewritten as
\begin{equation}\label{eqn:GRewrite5}
    G^{X, L, \vf}_{\beta, \vmu}(t) = (-1)^{\ell(\vmu)} z_{\vmu} \sum_{k \mid \beta, \vmu} \frac{\mu(k)}{k} \sum_{\vd: \pi(\vd) = \frac{\beta}{k}} F'_{\vd, \frac{\vmu}{k}}(q^k).
\end{equation}

Now we relate $Z'$ and $F'$ to the generating functions $Z^{\hY'}$ and $F^{\hY'}$ defined by the FTCY graph $\Gamma'$ introduced in Section \ref{sect:AuxGraph}. Let $\vd \in \Eff(\Gamma')$ and $\vlambda \in T_{\vd}$. At each $v_i$ viewed as a trivalent vertex in $\Gamma'$, we have $\vlambda^{v_i} = (\vlambda(-e_i), \emptyset, \emptyset)$. By the Frobenius character formula,
$$
    \cW_{\vlambda^{v_i}}(q) = \cW_{(\vlambda(-e_i), \emptyset, \emptyset)} = s_{\vlambda(-e_i)}(q^\rho) = \sum_{\mu^i \vdash \vd(\bar{e}_i)}\frac{\chi_{\vlambda(-e_i)}(\mu^i)}{z_\mu[\mu^i]},
$$
and this summation coincides with the expression of $\cW_{\vlambda^{v_i}}$ as a sum of contributions indexed by $\nu^{f_2(v_i)} = \mu^i, \nu^{v_i} = \nu^{f_1(v_i)} = \nu^{f_3(v_i)} = \emptyset$ described in Section \ref{sect:AuxGraph} (\cite[Lemma 6.2]{Konishi06b}). It follows that 
\begin{align*}
    Z^{\hY'}_{\vd}(q) & = \sum_{\mu^i \vdash \vd(\bar{e}_i)} \sum_{\vlambda \in T_{\vd}} \prod_{\bar{e} \in E(\Gamma)} (-1)^{(n^e + 1)\vd(\bar{e})} q^{\frac{\kappa_{\vlambda(e)}n^e}{2}} \prod_{v \in V^3(\Gamma)} \cW_{\vlambda^v}(q) \prod_{i = 1}^s \frac{\chi_{\vlambda(-e_i)}(\mu^i)}{z_{\mu^i} [\mu^i]}\\
    & = \sum_{\mu^i \vdash \vd(\bar{e}_i)} Z'_{\vd, \vmu}(q),
\end{align*}
and \eqref{eqn:KonishiZ} further implies that
$$
    Z'_{\vd, \vmu}(q) = \sum_{(\vnu_V, \vnu_F) \in D_{\vd, \vmu}} \frac{1}{z_{\vnu_V}z_{\vnu_F}} \sum_{W \in \Comb_{\Gamma'}^\bullet(\vnu_V, \vnu_F)} \cH(W)(q).
$$
The proof of \eqref{eqn:KonishiF} from \eqref{eqn:KonishiZ} in \cite[Proposition 6.13]{Konishi06b} in fact shows that\footnote{In more detail, and in the context of the proof of \cite[Proposition 6.13]{Konishi06b}, it suffices to introduce the formal variables $P$ to the definition of the map $\Psi$:
$$
    \frac{1}{|\Aut(G)|}\Psi(\bar{G}) = \sum_{\substack{W \in \Comb_{\Gamma'}^\bullet(\vnu_V, \vnu_F) \\ \bar{W} = G}} \frac{1}{z_{\vnu_V}z_{\vnu_F}}\cH(W)Q^{\vd}P_{(\nu^{(v_1, -e_1)}, \dots, \nu^{(v_s, -e_s)})}.
$$
The resulting map $\Psi$ is still grade-preserving and multiplicative. Our desired result follows from identifying the coefficients of $Q^{\vd}P_{\vmu}$ on the two sides of the exponential formula.}
$$
    F'_{\vd, \vmu}(q) = \sum_{(\vnu_V, \vnu_F) \in D_{\vd, \vmu}} \frac{1}{z_{\vnu_V}z_{\vnu_F}} \sum_{W \in \Comb_{\Gamma'}^\circ(\vnu_V, \vnu_F)} \cH(W)(q).
$$

We now return to \eqref{eqn:GRewrite5} and continue rewriting $G^{X, L, \vf}_{\beta, \vmu}(t)$ as
\begin{align*}
    G^{X, L, \vf}_{\beta, \vmu}(t) & = (-1)^{\ell(\vmu)} z_{\vmu} \sum_{k \mid \beta, \vmu} \sum_{k' \mid k} \sum_{\substack{\vd: \pi(\vd) = \frac{\beta}{k} \\ \gcd(\vd, \frac{\mu}{k})=1}} \frac{k'\mu(\frac{k}{k'})}{k} F'_{k'\vd, \frac{k'\vmu}{k}}(q^{\frac{k}{k'}})\\
    & = (-1)^{\ell(\vmu)} z_{\vmu} \sum_{k \mid \beta, \vmu} \sum_{\substack{\vd: \pi(\vd) = \frac{\beta}{k} \\ \gcd(\vd, \frac{\vmu}{k})=1}} \sum_{(\vnu_V, \vnu_F) \in D_{\vd, \frac{\vmu}{k}}} \sum_{k' \mid k}  \frac{k'\mu(\frac{k}{k'})}{kz_{k'\vnu_V}z_{k'\vnu_F}} \sum_{W \in \Comb_{\Gamma'}^\circ(\vnu_V, \vnu_F)} \cH(W_{(k')})(q^{\frac{k}{k'}}).
\end{align*}
Here, as in \cite[Section 6.7]{Konishi06b}, $W_{(k')} \in \Comb_{\Gamma'}^\circ(k'\vnu_V, k'\vnu_F)$ denotes the labeled graph obtained from $W$ by multiplying all vertex labels by $k'$. Moreover, $\gcd(\vd, \frac{\mu}{k})$ denotes the greatest common divisor of the set of all values $\vd(\bar{e})$ for $e \in E_c(\Gamma)$ and all parts of partitions in $\frac{\vmu}{k}$. The condition $\gcd(\vd, \frac{\mu}{k}) = 1$ implies that for any $(\vnu_V, \vnu_F) \in D_{\vd, \frac{\vmu}{k}}$, we have $\gcd(\vnu_V, \vnu_F) = 1$, which is the greatest common divisor of the set of all parts of partitions in $\vnu_V, \vnu_F$. Thus we may apply \cite[Proposition 6.14]{Konishi06b} to obtain that for any $k, \vd, (\vnu_V, \vnu_F)$,
$$
    t\sum_{k' \mid k}  \frac{k'\mu(\frac{k}{k'})}{kz_{k'\vnu_V}z_{k'\vnu_F}} \sum_{W \in \Comb_{\Gamma'}^\circ(\vnu_V, \vnu_F)} \cH(W_{(k')})(q^{\frac{k}{k'}}) \in \bQ[t].
$$
It then follows that $tG^{X, L, \vf}_{\beta, \vmu}(t) \in \bQ[t]$, proving the lemma. \qed

%% file: correspondence.tex
\section{Open/closed BPS correspondence and integrality}\label{sect:OpenClosed}
In this section, we specialize to the case of disk invariants of a single outer brane and extend the open/closed correspondence \cite{LY21,LY22} of Gromov-Witten invariants to a correspondence of BPS invariants (Theorem \ref{thm:BPSCorrespondence}). Using the integrality of the open BPS invariants (Theorem \ref{thm:OpenBPS}), we obtain the integrality of the closed BPS invariants (Corollary \ref{cor:ClosedBPS}).

\subsection{Specializing to disk invariants}\label{sect:Disk}
We first consider the specialization of results in Section \ref{sect:OpenBPS} to the case of disk invariants of a single outer brane. Let $s=1$ and we write $(L,f) = (L^1, f_1)$, $d = d_1 \in \bZ_{\ge 1}$, $B = B_1$, and $C = C_1$. We further consider the case $\mu = \mu^1 = (d)$. The open Gromov-Witten invariant
$$
    N^{X, L, f}_{0, \beta, (d)} \in \bQ
$$
is a virtual count of open stable maps from domains of arithmetic genus zero and with a single boundary component, and we refer to it as the degree-$(\beta, (d))$ \emph{disk invariant} of $(X, L, f)$.

In the above setting, the genus-zero limit $g_s \to 0$ (or equivalently $q \to 1$) of the LMOV resummation formula \eqref{eqn:LMOVResumGs} gives
\begin{equation}\label{eqn:LMOVOpen}
    N^{X, L, f}_{0, \beta, (d)} = - \sum_{k \mid \beta, d} \frac{n^{X, L, f}_{0,\frac{\beta}{k}, \left(\frac{d}{k}\right)}}{k^2}.
\end{equation}
By Theorem \ref{thm:OpenBPS}, the genus-zero, degree-$(\beta, (d))$ open BPS invariant of $(X, L, f)$ defined by the above is an integer for any $\beta, d$:
$$
    n^{X, L, f}_{0, \beta, (d)} \in \bZ.
$$

\subsection{Gromov-Witten correspondence}\label{sect:OpenClosedGW}
In \cite{LY21}, Liu and the author identified the disk invariants of the open geometry $(X,L,f)$ with the genus-zero Gromov-Witten invariants of a closed geometry on a smooth toric Calabi-Yau 4-fold $\tX$. As in \cite[Assumption 2.3]{LY21}, we make the following assumption.

\begin{assumption}\label{assump:Outer} \rm{
We assume that $L$ remains an outer brane in a (and thus any) toric Calabi-Yau semi-projective partial compactification of $X$.
}
\end{assumption}

In the case of a single outer brane ($s=1$), the FTCY graph $\Gamma$ constructed in Section \ref{sect:FTCY} in fact defines a relative Calabi-Yau 3-fold $(Y,D)$ such that $Y = X \sqcup D$ and $(\hY, \hD)$ is the formal completion of $(Y,D)$ along the toric 1-skeleton $Y^1$. Then $\tX = \Tot(\cO_Y(-D))$. The normal bundle of the new $\tT'$-invariant projective line $C$ in $\tX$ is
$$
    N_{C/\tX} \cong \cO_{\bP^1}(f) \oplus \cO_{\bP^1}(-f-1) \oplus \cO_{\bP^1}(-1).
$$

To provide more detail of the construction, we now summarize the description of $\tX$ in \cite[Section 2]{LY21} in terms of the toric data. Recall from Section \ref{sect:OpenGeometry} that the outer brane $L$ determines cones $\tau_1 \in \Sigma(2) \setminus \Sigma(2)_c$ and $\sigma_1 \in \Sigma(3)$. Now, let
$$
    b_1, \dots, b_R \in \{v \in N \mid \inner{\su_3, v} = 1\} \subset N
$$
be a listing of the primitive generators of the rays in $\Sigma$, labelled in such a way that:
\begin{itemize}
    \item $\tau_1$ is spanned by $\{b_2, b_3\}$;
    \item $\sigma_1$ is spanned by $\{b_1, b_2, b_3\}$;
    \item the orientation on $N$ determined by the basis $\{b_1, b_2, b_3\}$ agrees with the standard orientation.
\end{itemize}
Then, to construct $Y$ from $X$, we add an additional ray spanned by
$$
    b_{R+1} = -b_1 - fb_2 + (f+1)b_3 \quad \in \ker(\su_3) \subset N.
$$
Moreover, we add an additional 3-cone that is spanned by $\{b_2, b_3, b_{R+1}\}$, as well as all the faces of the cone. The resulting fan $\hSi$ is the toric fan of $Y$, and the divisor corresponding to the ray spanned by $b_{R+1}$ is $D = Y \setminus X$. The curve $C \cong \bP^1$ is the orbit $V(\tau_1)$ in $Y$, with
$$
    N_{C/Y} \cong \cO_{\bP^1}(f) \oplus \cO_{\bP^1}(-f-1).
$$
Finally, let $\tN := N \oplus \bZ \tv \cong \bZ^4$. We describe the toric fan $\tSi$ of $\tX$ in $\tN \otimes \bR \cong \bR^4$ by listing all the rays and 4-cones. There are $R+2$ rays in $\tSi$, spanned by $b_1, \dots, b_R$, and
$$
    \tb_{R+1} = b_{R+1} + b_3 + \tv, \qquad \tb_{R+2} = b_3 + \tv
$$
respectively. Extending $\su_3: N \to \bZ$ to $\tN$ by setting $\inner{\su_3,\tv} = 0$, we see that $\su_3$ pairs to 1 with all the primitive generators above, which is equivalent to that $\tX$ is Calabi-Yau. Each 3-cone $\sigma \in \Sigma(3)$ gives rise to a 4-cone in $\tSi$ that is spanned by $\sigma$ and $\tb_{R+2}$. Moreover, there is an additional $4$-cone spanned by $\{b_2, b_3, \tb_{R+1}, \tb_{R+2}\}$. The fan $\tSi$ consists of the 4-cones above and all their faces.

It follows from the construction that there is an isomorphism
$$
    \iota: H_2(X,L; \bZ) \to H_2(Y;\bZ) \to H_2(\tX;\bZ).
$$
Moreover, we consider the integral cohomology class
$$
    \tgamma := [\tD][\tD_2] \in H^4(\tX;\bZ)
$$
where $\tD$ (resp. $\tD_2$) is the toric divisor in $\tX$ corresponding to the ray spanned by $\tb_{R+1}$ (resp. $b_2$).

\begin{theorem}[\cite{LY21}]\label{thm:GWCorrespondence}
For the effective class $\tbeta:= \iota(\beta = \beta'+d[B])$, we have
$$
    N^{X, L, f}_{0, \beta, (d)} = N^{\tX}_{0, \tbeta}(\tgamma)
$$
where $N^{\tX}_{0, \tbeta}(\tgamma)$ is the genus-zero, degree-$\tbeta$ \emph{closed Gromov-Witten invariant} of $\tX$ with 1-pointed insertion $\tgamma$.
\end{theorem}

We note that since $\tX$ is non-compact, the invariant $N^{\tX}_{0, \tbeta}(\tgamma)$ is defined by localization with respect to the Calabi-Yau 3-torus $\tT'$ of $\tX$. We refer to \cite[Section 3]{LY21} for additional details.


\begin{remark}\label{rem:sp} \rm{
When $X$ is semi-projective, in \cite{LY22}, Liu and the author further showed that the disk invariants of $(X, L, f)$ can be identified with the genus-zero Gromov-Witten invariants of the toric Calabi-Yau semi-projective partial compactification of $\tX$, which we denote by $\tX^{\mathrm{sp}}$ here. Depending on the choice of framing $f$, $\tX^{\mathrm{sp}}$ is in general a toric Calabi-Yau 4-\emph{orbifold}. For our present purpose, consider the case where $\tX^{\mathrm{sp}}$ is still smooth. In the process of the partial compactification, the two toric divisors $\tD$ and $\tD_2$ may acquire an additional common torus fixed point in $\tX^{\mathrm{sp}}$. If this does not happen, we can still use the class $[\tD][\tD_2]$ in the open/closed correspondence as above. It this happens, we need to replace the insertion by an \emph{equivariant} cohomology class that restrict trivially to the new fixed point. In fact, the new fixed point corresponds to an Aganagic-Vafa outer brane in $X$ neighboring $L$, with framing depending on $f$. If we still insert the non-equivariant class $[\tD][\tD_2]$ for $\tX^{\mathrm{sp}}$, in localization we obtain contributions from both fixed points in $\tD \cap \tD_2$ which correspond to disk invariants of the two framed outer branes in $X$ respectively with the same boundary winding number.
}\end{remark}

\begin{example}\label{ex:C3} \rm{
Consider the basic example $X = \bC^3$ as in \cite[Section 1.2]{LY21}, \cite[Section 2.6.1]{LY22}. For an outer brane $L$ with framing $f$, we have
$$
    Y = \Tot(\cO_{\bP^1}(f) \oplus \cO_{\bP^1}(-f-1)), \qquad \tX = \Tot(\cO_{\bP^1}(f) \oplus \cO_{\bP^1}(-f-1) \oplus \cO_{\bP^1}(-1)).
$$
Consider the case $f<-1$, where $\tX$ is not semi-projective. The partial compactification produces
$$
    \tX^{\mathrm{sp}} = \Tot(\cO_{\bP(1,1,-f-1)}(f) \oplus \cO_{\bP(1,1,-f-1)}(-1)).
$$
Now we specialize to $f = -2$, where $\tX^{\mathrm{sp}} = \Tot(\cO_{\bP^2}(-2) \oplus \cO_{\bP^2}(-1))$. The construction is illustrated in Figure \ref{fig:ExC3}, where the fan of $X$ is the cone over the triangle on the left and the fan of $\tX^{\mathrm{sp}}$ is the cone over the triangulated 3-dimensional polytope on the right. In this case, the same 4-fold $\tX^{\mathrm{sp}}$ is reached if we start the construction instead with the neighboring outer brane in $X$ corresponding to the 2-cone spanned by $\{b_1, b_2\}$ with framing $1$. The closed invariant of $\tX^{\mathrm{sp}}$ with non-equivariant insertion $[\tD][\tD_2]$ has contribution from the disk invariants of the two framed outer branes. 
}
\end{example}

\begin{figure}[h]
\begin{center}
    \begin{tikzpicture}[scale=1.5]
        \coordinate (1) at (-3.2, 0.1);
        \coordinate (2) at (-4, 0.9);        
        \coordinate (3) at (-4, 0.1);
                
        \node at (1) {$\bullet$};
        \node at (2) {$\bullet$};
        \node at (3) {$\bullet$};

        \node at (-2.9, 0.1) {$b_3$};
        \node at (-4, 1.2) {$b_1$};
        \node at (-4.3, 0.1) {$b_2$};

        \draw[ultra thick] (-3.6, 0.2) -- (-3.6, 0);
        \node[below] at (-3.6, 0) {$(L, f = -2)$};

        \draw (1) -- (2) -- (3) -- (1);

        \node at (-3.6, -1) {$X = \bC^3$};

        \coordinate (11) at (2, 0);
        \coordinate (22) at (1.6, 0.7);        
        \coordinate (33) at (1, 0);
        \coordinate (44) at (0.4, 0.3);        
        \coordinate (55) at (1, 1);

        \node at (11) {$\bullet$};
        \node at (22) {$\bullet$};
        \node at (33) {$\bullet$};
        \node at (44) {$\bullet$};
        \node at (55) {$\bullet$};

        \node at (2.3, 0) {$b_3$};
        \node at (1.9, 0.7) {$b_1$};
        \node at (1, -0.3) {$b_2$};
        \node at (0.1, 0.3) {$\tb_4$};
        \node at (1, 1.3) {$\tb_5$};

        \draw (11) -- (33) -- (44) -- (55) -- (22) -- (11) -- (44);
        \draw (11) -- (55);
        \draw[dashed] (55) -- (33) -- (22) -- (44);

        \node at (1.3, -1) {$\tX^{\mathrm{sp}} = \Tot(\cO_{\bP^2}(-2) \oplus \cO_{\bP^2}(-1))$};
    \end{tikzpicture}
\end{center}

    \caption{Example of $X = \bC^3$ and $f=-2$.}
    \label{fig:ExC3}
\end{figure}

\subsection{BPS correspondence}\label{sect:OpenClosedBPS}
For Calabi-Yau 4-folds, Klemm-Pandharipande \cite{KP08} expressed integral structures of genus-zero Gromov-Witten invariants by a generalization of the Aspinwall-Morrison multiple covering formula \cite{AM93}. For invariants with 1-pointed insertions, the resummation of \cite{KP08} is
\begin{equation}\label{eqn:KPClosed}
    N^{\tX}_{0, \tbeta}(\tgamma) = \sum_{k \mid \tbeta} \frac{n^{\tX}_{0,\frac{\tbeta}{k}}(\tgamma)}{k^2}.
\end{equation}
Here for $k \in \bZ_{\ge 1}$, we say that $k \mid \tbeta$ if $\frac{\tbeta}{k} \in H_2(\tX;\bZ)$. Note that if $\tbeta = \iota(\beta = \beta' + d[B])$, then $k \mid \tbeta$ if and only if $k \mid \beta, d$. We refer to the coefficient
$$
    n^{\tX}_{0, \tbeta}(\tgamma) \in \bQ
$$
as the genus-zero, degree-$\tbeta$ \emph{closed BPS invariant} of $\tX$.

By a direct comparison of the resummations \eqref{eqn:LMOVOpen}, \eqref{eqn:KPClosed}, we see that Theorem \ref{thm:GWCorrespondence} immediately implies the following open/closed correspondence of BPS invariants.

\begin{theorem}\label{thm:BPSCorrespondence}
For $\tbeta = \iota(\beta = \beta'+d[B])$, we have
$$
    n^{X, L, f}_{0, \beta, (d)} = -n^{\tX}_{0, \tbeta}(\tgamma).
$$
\end{theorem}

\subsection{Integrality of closed BPS invariants}\label{sect:ClosedBPS}

As a result of the integrality of open BPS invariants (Theorem \ref{thm:OpenBPS}), Theorem \ref{thm:BPSCorrespondence} implies the integrality of closed BPS invariants of $\tX$ with 1-pointed insertion $\tgamma$, verifying the conjecture of Klemm-Pandharipande \cite[Conjecture 0]{KP08}.

\begin{corollary}\label{cor:ClosedBPS}
For $\tbeta = \iota(\beta = \beta'+d[B])$, we have
$$
    n^{\tX}_{0, \tbeta}(\tgamma) \in \bZ.
$$
\end{corollary}

As discussed in Section \ref{sect:Intro}, \cite[Conjecture 0]{KP08} has been extensively studied in the literature. For compact Calabi-Yau 4-folds, it has been verified in examples in \cite{KP08,CMT18,Cao20,CMT22,COT22,COT24} and in general by \cite{IP18}. In the non-compact setting, a general proof is yet to be given, and to the best of the author's knowledge the known examples are the following:
\begin{itemize}
    \item Local curve $\Tot(\cL_1 \oplus \cL_2 \oplus \cL_3 \to C)$, where $\cL_1, \cL_2, \cL_3$ are line bundles on a smooth projective curve $C$ with $K_C \cong \cL_1 \otimes \cL_2 \otimes \cL_3$, in the cases
    
    \begin{itemize}
        \item[$\circ$] (\cite{CMT18}) $C$ has genus at least 1;
        
        \item[$\circ$] (\cite{CMT18}) $C = \bP^1$ and the $\cL_i$'s have low degree.  
    \end{itemize}

    \item Local surface $\Tot(\cL_1 \oplus \cL_2 \to S)$, where $\cL_1, \cL_2$ are line bundles on a smooth projective surface $S$ with $K_S \cong \cL_1 \otimes \cL_2$, in the cases
    
    \begin{itemize}
        \item[$\circ$] (\cite{KP08}) $S = \bP^2$, $\cL_1 = \cO_{\bP^2}(-1)$, $\cL_2 = \cO_{\bP^2}(-2)$;
        
        \item[$\circ$] (\cite{KP08,CMT22,CKM22}) $S$ is the Hirzebruch surface $\bP^1 \times \bP^1$, $F_1$, or $F_2$, and $\cL_1, \cL_2$ have certain low degrees;
        
        \item[$\circ$] (\cite{CMT18,CMT22}) $S$ is a toric del Pezzo surface, $\cL_1 = \cO_S$, $\cL_2 = K_S$;
        
        \item[$\circ$] (\cite{CMT18,CMT22}) $S$ is a rational elliptic surface, $\cL_1 = \cO_S$, $\cL_2 = K_S$, and the curve class is primitive;
        
        \item[$\circ$] (\cite{BBvG20,BS23,Schuler24}) $(S \mid D_1 + D_2)$ is a two-component Looijenga pair (see e.g. \cite[Section 2]{BBvG20}), $\cL_1 = \cO_S(-D_1)$, $\cL_2 = \cO_S(-D_2)$.\footnote{This covers certain toric cases above as well as some non-toric cases. We refer to \cite[Section 2]{BBvG20} and the references therein for a classification of the deformation types of Looijenga pairs.}
        
    \end{itemize}

    \item Local 3-fold $\Tot(K_W)$, where $W$ is a smooth projective 3-fold, in the cases
    
    \begin{itemize}
        \item[$\circ$] (\cite{KP08}) $W = \bP^3$;
        
        \item[$\circ$] (\cite{Cao20}) $W$ is a Fano hypersurface in $\bP^4$, and the curve class is the line class;
        
        \item[$\circ$] (\cite{Cao20}) $W = S \times \bP^1$ for a toric del Pezzo surface $S$.
    \end{itemize}

    \item (\cite{COT22,COT24}) $\Tot(T^*\bP^2)$, as an example of a holomorphic symplectic variety.

\end{itemize}

By Corollary \ref{cor:ClosedBPS}, the toric Calabi-Yau 4-folds $\tX$ arising from the open/closed correspondence can be added to the above list of non-compact examples. From the constructions, $\tX = \Tot(K_Y)$ where the (non-compact) toric 3-fold $Y$ may arise from an \emph{arbitrary} toric Calabi-Yau 3-fold $X$ and is not necessarily a local curve or local surface.

\begin{remark} \rm{
Following Remark \ref{rem:sp}, when $X$ is semi-projective and the resulting partial compactification $\tX^{\mathrm{sp}}$ is smooth, Corollary \ref{cor:ClosedBPS} can be directly generalized to $\tX^{\mathrm{sp}}$ with the non-equivariant insertion $[\tD][\tD_2]$, therefore providing yet another collection of examples. As observed in \cite[Section 0.5]{Schuler24}, if $(X, L, f)$ is the open geometry arising from a two-component Looijenga pair $(S \mid D_1 + D_2)$, as constructed in \cite{BBvG20,BS23,vGNS23,Schuler24}, then $\tX^{\mathrm{sp}}$ is deformation equivalent to the local surface $\Tot(\cO_S(-D_1) \oplus \cO_S(-D_2))$ mentioned in the list above.
}    
\end{remark}

%% file: appdx.tex
\appendix

\section{Topological vertex formulas}\label{appdx:TopVertex}
In this section, we supply the details of the expression \eqref{eqn:TopVertex} by clarifying the signs in \cite{LLLZ09} from equation (7-9) to Proposition 7.4. We adopt the definitions and notations in \cite{LLLZ09} without repeating them here and we fix a choice of $\sqrt{-1}$ as before. Our formulas are based on the statements and equations up to (7-8). Note that (7-8) directly implies (7-10).

We first correct (7-9) as follows:
\begin{equation}\label{eqn:7-9}
\begin{aligned}
    & \sqrt{-1}^{\ell(\vmu)} G^\bullet_{\vmu}(\lambda; \fp(e_1), \fp(e_2), \fp(e_3))\\
    & = \sqrt{-1}^{\ell(\vmu)} \sum_{|\nu^i| = |\mu^i|} \tF^\bullet_{\vnu}(\lambda; 0) \prod_{i=1}^3 z_{\nu^i} \Phi^\bullet_{\nu^i, \mu^i} \left(\sqrt{-1}\frac{\fl_0(e_i)}{\fp(e_i)}\lambda \right)\\
    & = \sqrt{-1}^{\ell(\vmu)} \sum_{|\nu^i| = |\mu^i|} (-1)^{|\vnu|}\sqrt{-1}^{\ell(\vnu)}F^\bullet_{\vnu}(\lambda; 0) \prod_{i=1}^3 z_{\nu^i} \Phi^\bullet_{\nu^i, \mu^i} \left(\sqrt{-1}\frac{\fl_0(e_i)}{\fp(e_i)}\lambda \right) && \qquad (\text{by (6-6)})\\
    & = (-1)^{\sum_{i=1}^3 \vd(\bar{e}_i)}  \sum_{|\nu^i| = |\mu^i|} F^\bullet_{\vnu}(\lambda; 0) \prod_{i=1}^3 \sqrt{-1}^{\ell(\mu^i) + \ell(\nu^i)}z_{\nu^i} \Phi^\bullet_{\nu^i, \mu^i} \left(\sqrt{-1}\frac{\fl_0(e_i)}{\fp(e_i)}\lambda \right)\\
    & = (-1)^{\sum_{i=1}^3 \vd(\bar{e}_i)} \sum_{|\nu^i| = |\mu^i|} F^\bullet_{\vnu}(\lambda; 0) \prod_{i=1}^3 \sqrt{-1}^{(-\ell(\mu^i)-\ell(\nu^i))}z_{\nu^i} \Phi^\bullet_{\nu^i, \mu^i} \left(-\sqrt{-1}\frac{\fl_0(e_i)}{\fp(e_i)}\lambda \right).
\end{aligned}
\end{equation}

Now we deduce (7-12) from (7-10) using the correction \eqref{eqn:7-9} of (7-9) above:
\begin{equation}\label{eqn:7-12}
    \begin{aligned}
        F^{\bullet \Gamma}_{\vd, \vmu}(\lambda;u_1, u_2)
        = & \sum_{|\rho^{\bar{e}}| = d^{\bar{e}}} \prod_{\bar{e} \in E(\Gamma)} (-1)^{n^ed^{\bar{e}}}z_{\rho^{\bar{e}}} \prod_{v \in V_3(\Gamma)} \sqrt{-1}^{\ell(\vrho^v)}G^\bullet_{\vrho^v}(\lambda;\bw_v) \\
        & \cdot \prod_{v \in V_1(\Gamma), v_1(e) = v} (-1)^{d^{\bar{e}}}\sqrt{-1}^{\ell(\rho^{\bar{e}}) + \ell(\mu^v)}\Phi^\bullet_{\rho^{\bar{e}}, \mu^v} \left(\sqrt{-1}\frac{\ff(e)}{\fp(e)}\lambda\right)\\
        = & \sum_{|\rho^{\bar{e}}| = d^{\bar{e}}} \prod_{\bar{e} \in E(\Gamma)} (-1)^{n^ed^{\bar{e}}}z_{\rho^{\bar{e}}} \prod_{v \in V_3(\Gamma)} (-1)^{|\vrho^v|}  \sqrt{-1}^{(-\ell(\vrho^v) -\ell(\vnu^v))} \\
        &  \cdot \sum_{|\nu^{v,i}| = |\rho^{v,i}|}F^\bullet_{\vnu^v}(\lambda; 0) \prod_{i=1}^3 z_{\nu^{v,i}} \Phi^\bullet_{\nu^{v,i}, \rho^{v,i}} \left(-\sqrt{-1}\frac{\fl_0(e_i)}{\fp(e_i)}\lambda \right)\\
        & \cdot \prod_{v \in V_1(\Gamma), v_1(e) = v} (-1)^{d^{\bar{e}}}\sqrt{-1}^{\ell(\rho^{\bar{e}}) + \ell(\mu^v)}\Phi^\bullet_{\rho^{\bar{e}}, \mu^v} \left(\sqrt{-1}\frac{\ff(e)}{\fp(e)}\lambda\right) \qquad \qquad \qquad (\text{by \eqref{eqn:7-9}})\\
        = & \sum_{\vnu \in P_{\vd, \vmu}} \prod_{v \in V_3(\Gamma)} F^\bullet_{\vnu^v}(\lambda; 0) z_{\vnu^v} \prod_{\bar{e} \in E(\Gamma)} \cE_{\bar{e}}
    \end{aligned}
\end{equation}
where the term $ \prod_{\bar{e} \in E(\Gamma)} \cE_{\bar{e}}$ above is a product of contributions from edges $\bar{e} \in E(\Gamma)$. Here, for an edge $\bar{e}$ between two trivalent vertices $v = v_0(e), v' = v_1(e) \in V_3(\Gamma)$, the contribution is
\begin{align*}
    \cE_{\bar{e}} & = (-1)^{n^ed^{\bar{e}}}\sum_{\rho \vdash d^{\bar{e}}} \sqrt{-1}^{(-\ell(\nu^e) - \ell(\nu^{-e}) - 2\ell(\rho))}\Phi^\bullet_{\nu^e, \rho}\left(-\sqrt{-1}\frac{\fl_0(e)}{\fp(e)}\lambda\right) z_\rho \Phi^\bullet_{\nu^{-e}, \rho}\left(-\sqrt{-1}\frac{\fl_0(-e)}{\fp(-e)}\lambda\right)\\
    & = (-1)^{n^ed^{\bar{e}}}\sum_{\rho \vdash d^{\bar{e}}} \sqrt{-1}^{(-\ell(\nu^e) - \ell(\nu^{-e}) - 2\ell(\rho))}(-1)^{\ell(\nu^e) + \ell(\rho)}\Phi^\bullet_{\nu^e, \rho}\left(\sqrt{-1}\frac{\fl_0(e)}{\fp(e)}\lambda\right) z_\rho \Phi^\bullet_{\nu^{-e}, \rho}\left(\sqrt{-1}\frac{\fl_0(-e)}{\fp(e)}\lambda\right)\\
    & = (-1)^{n^ed^{\bar{e}}}\sum_{\rho \vdash d^{\bar{e}}} \sqrt{-1}^{\ell(\nu^e) - \ell(\nu^{-e})}\Phi^\bullet_{\nu^e, \rho}\left(\sqrt{-1}\frac{\fl_0(e)}{\fp(e)}\lambda\right) z_\rho \Phi^\bullet_{\nu^{-e}, \rho}\left(\sqrt{-1}\frac{\fl_0(-e)}{\fp(e)}\lambda\right)\\
    & = (-1)^{n^ed^{\bar{e}}} \sqrt{-1}^{\ell(\nu^e) - \ell(\nu^{-e})}\Phi^\bullet_{\nu^e, \nu^{-e}}\left(\sqrt{-1}n^e\lambda\right)
\end{align*}
where the last equality follows from (2-9) applied with
$$
    \fl_0(e) + \fl_0(-e) = \fl_0(e) - \fl_1(e) + \fp(e) = n^e\fp(e).
$$
For an edge $\bar{e}$ between a trivalent vertex $v' = v_0(e) \in V_3(\Gamma)$ and a univalent vertex $v = v_1(e) \in V_1(\Gamma)$, the contribution is
\begin{align*}
    \cE_{\bar{e}} & = (-1)^{n^ed^{\bar{e}}}\sum_{\rho \vdash d^{\bar{e}}} \sqrt{-1}^{(-\ell(\nu^e) + \ell(\mu^v))}\Phi^\bullet_{\nu^e, \rho}\left(-\sqrt{-1}\frac{\fl_0(e)}{\fp(e)}\lambda\right) z_\rho \Phi^\bullet_{\mu^v, \rho}\left(\sqrt{-1}\frac{\ff(e)}{\fp(e)}\lambda\right)\\
    & = (-1)^{n^ed^{\bar{e}}}\sum_{\rho \vdash d^{\bar{e}}} \sqrt{-1}^{(-\ell(\nu^e) + \ell(\mu^v))}(-1)^{\ell(\nu^e) + \ell(\mu^v)}\Phi^\bullet_{\nu^e, \rho}\left(\sqrt{-1}\frac{\fl_0(e)}{\fp(e)}\lambda\right) z_\rho \Phi^\bullet_{\mu^v, \rho}\left(-\sqrt{-1}\frac{\ff(e)}{\fp(e)}\lambda\right)\\
    & = (-1)^{n^ed^{\bar{e}}} \sqrt{-1}^{\ell(\nu^e) - \ell(\mu^v)}\Phi^\bullet_{\nu^e, \mu^v}\left(\sqrt{-1}n^e\lambda\right)\\
    & = (-1)^{n^ed^{\bar{e}}} \sqrt{-1}^{\ell(\nu^e) - \ell(\nu^{-e})}\Phi^\bullet_{\nu^e, \nu^{-e}}\left(\sqrt{-1}n^e\lambda\right)
\end{align*}
where the second-to-last equality follows from (2-9) applied with
$$
    \fl_0(e) - \ff(e) = n^e\fp(e)
$$
and the last equality follows from $\nu^{-e} = \mu^v$. Therefore, \eqref{eqn:7-12} is identical to (7-12).

Lastly, we use (2-8), (7-12), and (7-13) to derive the following correction of Proposition 7.4:
\begin{equation}\label{eqn:Prop7-4}
\begin{aligned}
    F^{\bullet\Gamma}_{\vd, \vmu} = &  \sum_{\vnu \in P_{\vd, \vmu}} \prod_{v \in V_3(\Gamma)} F^\bullet_{\vnu^v}(\lambda; 0) z_{\vnu^v} \prod_{\bar{e} \in E(\Gamma)} (-1)^{n^ed^{\bar{e}}} \sqrt{-1}^{\ell(\nu^e) - \ell(\nu^{-e})}\Phi^\bullet_{\nu^e, \nu^{-e}}\left(\sqrt{-1}n^e\lambda\right) \qquad (\text{by (7-12)})\\
    = & \sum_{\vnu \in P_{\vd, \vmu}} \prod_{v \in V_3(\Gamma)} \frac{(-1)^{|\vnu^v|}}{\sqrt{-1}^{\ell(\vnu^v)}}\sum_{|\xi^{v,i}| = |\nu^{v,i}|} \tC_{\vxi^v}(\lambda)\prod_{i=1}^3 \chi_{\xi^{v,i}}(\nu^{v,i})\\
     & \cdot \prod_{\bar{e} \in E(\Gamma)} (-1)^{n^ed^{\bar{e}}} \sqrt{-1}^{\ell(\nu^e) - \ell(\nu^{-e})} \sum_{\rho^{e} \vdash d^{\bar{e}}} e^{\sqrt{-1}n^e\kappa_{\rho^e}\lambda/2} \frac{\chi_{\rho^e}(\nu^e)}{z_{\nu^e}}\frac{\chi_{\rho^e}(\nu^{-e})}{z_{\nu^{-e}}} \quad (\text{by (2-8) and (7-13)})\\
     = & \sum_{\vnu \in P_{\vd, \vmu}} \prod_{v \in V_3(\Gamma)} \frac{(-1)^{|\vnu^v|}}{\sqrt{-1}^{\ell(\vnu^v)}}\sum_{|\xi^{v,i}| = |\nu^{v,i}|} \tC_{\vxi^v}(\lambda)\prod_{i=1}^3 \chi_{\xi^{v,i}}(\nu^{v,i})\\
     & \cdot \prod_{\bar{e} \in E(\Gamma)} (-1)^{(n^e+1)d^{\bar{e}}} \sqrt{-1}^{\ell(\nu^e) + \ell(\nu^{-e})} \sum_{\rho^{e} \vdash d^{\bar{e}}} e^{\sqrt{-1}n^e\kappa_{\rho^e}\lambda/2} \frac{\chi_{\rho^e}(\nu^e)}{z_{\nu^e}}\frac{\chi_{(\rho^e)^t}(\nu^{-e})}{z_{\nu^{-e}}} \\
     = & \sum_{\vxi, \vrho} \prod_{\bar{e} \in E(\Gamma)} (-1)^{(n^e+1)d^{\bar{e}}} e^{\sqrt{-1}n^e\kappa_{\rho^e}\lambda/2} \prod_{v \in V_3(\Gamma)} (-1)^{|\vxi^v|} \tC_{\vxi^v}(\lambda) \\
     & \cdot \prod_{v = v_0(e) \in V_3(\Gamma)} \sum_{\nu^e \vdash d^{\bar{e}}} \frac{\chi_{\xi^e}(\nu^e)\chi_{\rho^e}(\nu^e)}{z_{\nu^e}} \prod_{v = v_0(e) \in V_1(\Gamma)} \frac{\chi_{\rho^e}(\mu^v)}{z_{\mu^v}}\sqrt{-1}^{\ell(\mu^v)} \qquad (\text{where $\rho^{-e} := (\rho^e)^t$})\\
     = & \sum_{\vxi, \vrho} \prod_{\bar{e} \in E(\Gamma)} (-1)^{(n^e+1)d^{\bar{e}}} e^{\sqrt{-1}n^e\kappa_{\rho^e}\lambda/2} \prod_{v \in V_3(\Gamma)} (-1)^{|\vxi^v|} \tC_{\vxi^v}(\lambda) \\
     & \cdot \prod_{v = v_0(e) \in V_3(\Gamma)} \delta_{\xi^e, \rho^e} \prod_{v = v_0(e) \in V_1(\Gamma)} \frac{\chi_{\rho^e}(\mu^v)}{z_{\mu^v}}\sqrt{-1}^{\ell(\mu^v)} \qquad \qquad  (\text{by orthogonality of characters})\\
     = &  \sum_{\vrho \in T_{\vd, \vmu}} \prod_{\bar{e} \in E(\Gamma)} (-1)^{(n^e+1)d^{\bar{e}}} e^{\sqrt{-1}n^e\kappa_{\rho^e}\lambda/2} \prod_{v \in V_3(\Gamma)} \tC_{\vrho^v}(\lambda) \prod_{v = v_0(e) \in V_1(\Gamma)} \frac{\chi_{\rho^e}(\mu^v)}{z_{\mu^v}}\sqrt{-1}^{\ell(\mu^v)}(-1)^{|\mu^v|}.
\end{aligned}
\end{equation}
Equation \eqref{eqn:Prop7-4} differs from Proposition 7.4 in two ways:
\begin{itemize}
    \item For each $v \in V_1(\Gamma)$, there is an additional sign of $(-1)^{|\mu^v|+\ell(\mu^v)}$.
    
    \item For each $\bar{e} \in E(\Gamma)$, the sign in the power $e^{\sqrt{-1}n^e\kappa_{\rho^e}\lambda/2}$ is positive rather than negative. 
\end{itemize}
Moreover, \eqref{eqn:TopVertex} directly follows from \eqref{eqn:Prop7-4}.

%% file: openclosedBPS.bbl
\begin{thebibliography}{AA}


\bibitem{AKMV03} M. Aganagic, A. Klemm, M. Mari\~no, C. Vafa, 
``The topological vertex,'' 
Commun. Math. Phys. {\bf 254} (2005), 425--478.

\bibitem{AKV02} M. Aganagic, A. Klemm, and C. Vafa, 
``Disk instantons, mirror symmetry and the duality web,'' 
Z. Naturforsch. A \textbf{57} (2002), no. 1-2, 1--28.

\bibitem{AV00} M. Aganagic, C. Vafa, 
``Mirror symmetry, D-branes and counting holomorphic discs,'' 
\texttt{arXiv:hep-th/0012041}.

\bibitem{AL23} K. Aleshkin, C.-C. M. Liu, 
``Open/closed correspondence and extended LG/CY correspondence for quintic threefolds,''
{\tt arXiv:2309.14628}.

\bibitem{AM93}  P. Aspinwall, D.  Morrison, 
``Topological field theory and rational curves,"
Commun. Math. Phys. {\bf 151} (1993), no. 2, 245--262.





\bibitem{BKMP09}  V. Bouchard, A. Klemm, M. Mari\~{n}o, S. Pasquetti,
``Remodeling the B-model,''
Commun. Math. Phys. {\bf 287} (2009), no. 1, 117--178.

\bibitem{BKMP10} V. Bouchard, A. Klemm, M. Mari\~{n}o, S. Pasquetti,
``Topological open strings on orbifolds,''
Commun. Math. Phys. {\bf 296} (2010), no. 3, 589--623.


\bibitem{BBvG20} P. Bousseau, A. Brini, M. van Garrel,
``Stable maps to Looijenga pairs,''
Geom. Topol. {\bf 28} (2024), no. 1, 393--496.





\bibitem{BS23} A. Brini, Y. Schuler, 
``On quasi-tame Looijenga pairs'',
Commun. Number Theory Phys. {\bf 17} (2023), no. 2, 313--341.

\bibitem{BP01} J. Bryan, R. Pandharipande, ``BPS states of curves in Calabi-Yau 3-folds,'' Geom. Topol. {\bf 5} (2001), no. 1, 287--318.

\bibitem{BP08} J. Bryan, R. Pandharipande, 
``The local Gromov-Witten theory of curves,'' 
J. Amer. Math. Soc. {\bf 21} (2008), no. 1, 101--136.

\bibitem{Cao20} Y. Cao,
``Genus zero Gopakumar-Vafa type invariants for Calabi-Yau 4-folds II: Fano 3-folds,'' 
Commun. Contemp. Math. {\bf 22} (2020), no. 7, 1950060.

\bibitem{CKM22} Y. Cao, M. Kool, S. Monavari, 
``Stable pair invariants of local Calabi-Yau 4-folds,'' 
Internat. Math. Res. Notices (2022), no. 6, 4753--4798.

\bibitem{CMT18} Y. Cao, D. Maulik, Y. Toda,
``Genus zero Gopakumar-Vafa type invariants for Calabi-Yau 4-folds,'' 
Adv. Math. {\bf 338} (2018), 41--92.

\bibitem{CMT22} Y. Cao, D. Maulik, Y. Toda,
``Stable pairs and Gopakumar-Vafa type invariants for Calabi-Yau 4-folds,'' 
J. Eur. Math. Soc. {\bf 24} (2022), no. 2, 527--581.

\bibitem{COT22} Y. Cao, G. Oberdieck, Y. Toda,
``Stable pairs and Gopakumar-Vafa type invariants on holomorphic symplectic 4-folds,''
Adv. Math. \textbf{408} (2022), Paper No. 108605, 44 pp.

\bibitem{COT24} Y. Cao, G. Oberdieck, Y. Toda,
``Gopakumar-Vafa type invariants of holomorphic symplectic 4-folds,''
Commun. Math. Phys. {\bf 405} (2024), 26.






\bibitem{DIW21} A. Doan, E. Ionel, T. Walpuski,
``The Gopakumar-Vafa finiteness conjecture,'' 
\texttt{arXiv:2103.08221}.

\bibitem{DW19} A. Doan, T. Walpuski,
``Counting embedded curves in symplectic 6-manifolds,''
Comment. Math. Helv. {\bf 98} (2023), no.4, 693--769.

\bibitem{Efimov12} A. Efimov,
``Cohomological Hall algebra of a symmetric quiver,'' 
Compos. Math. {\bf 148} (2012), 1133--1146.

\bibitem{EKL20a} T. Ekholm, P. Kucharski, P. Longhi, 
``Multi-cover skeins, quivers, and 3d $\mathcal{N} = 2$ dualities,'' 
JHEP {\bf 2020} (2020), 018.

\bibitem{EKL20b} T. Ekholm, P. Kucharski, P. Longhi, 
``Physics and geometry of knots-quivers correspondence,'' 
Commun. Math. Phys. {\bf 379} (2020), 361--415. 

\bibitem{EO15} E. Eynard, N. Orantin,
``Computation of open Gromov-Witten invariants for toric Calabi-Yau 3-folds by
topological recursion, a proof of the BKMP conjecture,''
Commun. Math. Phys. {\bf 337} (2015), no. 2, 483--567.

\bibitem{FL13} B. Fang, C.-C. M. Liu,
``Open Gromov-Witten invariants of toric Calabi-Yau 3-folds,''
Commun. Math. Phys. {\bf 323} (2013), 285--328.

\bibitem{FLZ20} B. Fang, C.-C. M. Liu, Z. Zong,
``On the Remodeling Conjecture for toric Calabi-Yau 3-orbifolds,''
J. Amer. Math. Soc. {\bf 33} (2020), no. 1, 135--222.



\bibitem{vGGR19} M. van Garrel, T. Graber, H. Ruddat, ``Local Gromov-Witten invariants are log invariants,'' Adv. Math. \textbf{350} (2019), 860--876.

\bibitem{vGNS23} M. van Garrel, N. Nabijou, Y. Schuler,
``Gromov-Witten theory of bicyclic pairs,''
{\tt arXiv:2310.06058}.


\bibitem{GV98a} R. Gopakumar, C. Vafa,
``M-theory and topological strings-I,'' 
\texttt{arXiv:hep-th/9809187}.

\bibitem{GV98b} R. Gopakumar, C. Vafa,
``M-theory and topological strings-II,'' 
\texttt{arXiv:hep-th/9812127}.

\bibitem{GV99} R. Gopakumar, C. Vafa,
``On the gauge theory/geometry correspondence,'' 
Adv. Theor. Math. Phys. {\bf 3} (1999), 1415--1443.


\bibitem{GRZ22} T. Gr\"{a}fnitz, H. Ruddat, E. Zaslow,
``The proper Landau-Ginzburg potential is the open mirror map,''
Adv. Math. \textbf{447} (2024), Paper No. 109639, 69 pp.


\bibitem{HST01} S. Hosono, M.-H. Saito, A. Takahashi, 
``Relative Lefschetz action and BPS state counting,'' 
Internat. Math. Res. Notices {\bf 15} (2001), 783--816.

\bibitem{IP18} E. Ionel, T. Parker,
``The Gopakumar-Vafa formula for symplectic manifolds,'' 
Ann. of Math. {\bf 187} (2018), 1--64.

\bibitem{Katz08} S. Katz, 
``Genus zero Gopakumar-Vafa invariants of contractible curves,'' 
J. Differential Geom. {\bf 79} (2008), 185--195.

\bibitem{KL01} S. Katz, C.-C. M. Liu, ``Enumerative geometry of stable maps with Lagrangian boundary conditions and multiple covers of the disc,'' Adv. Theor. Math. Phys. {\bf 5} (2001), no. 1, 1--49.

\bibitem{KP08} A. Klemm, R. Pandharipande,
``Enumerative geometry of Calabi-Yau 4-folds,'' 
Commun. Math. Phys. {\bf 281} (2008), 621--653.

\bibitem{Konishi06a} Y. Konishi, 
``Pole structure of topological string free energy,'' 
Publ. RIMS, Kyoto Univ. {\bf 42} (2006), 173--219.

\bibitem{Konishi06b} Y. Konishi, 
``Integrality of Gopakumar-Vafa invariants of toric Calabi-Yau threefolds,'' Publ. RIMS, Kyoto Univ. {\bf 42} (2006), 605--648.

\bibitem{KRSS17} P. Kucharski, M. Reineke, M. Sto\v{s}i\'{c}, P. Su\l{}kowski, 
``BPS states, knots and quivers,'' 
Phys. Rev. {\bf D 96} (2017), 121902.

\bibitem{KRSS19} P. Kucharski, M. Reineke, M. Sto\v{s}i\'{c}, P. Su\l{}kowski, 
``Knots-quivers correspondence,'' 
Adv. Theor. Math. Phys. {\bf 23} (2019), 1849--1902.

\bibitem{LM00} J. M. F. Labastida, M. Mari\~no,
``Polynomial invariants for torus knots and topological strings,'' 
Commun. Math. Phys. {\bf 217} (2001), 423--449.

\bibitem{LMV00} J. M. F. Labastida, M. Mari\~no, C. Vafa,
``Knots, links and branes at large $N$,'' 
JHEP {\bf 11} (2000), 007.
  



\bibitem{LLLZ09} J. Li, C.-C. M. Liu, K. Liu, J. Zhou, ``A mathematical theory of the topological vertex,'' Geom. Topol. {\bf 13} (2009), no. 1, 527--621.






\bibitem{LY21} C.-C. M. Liu, S. Yu, 
``Open/closed correspondence via relative/local correspondence,''
Adv. Math. \textbf{410} (2022), Paper No. 108696, 43 pp.

\bibitem{LY22} C.-C. M. Liu, S. Yu, 
``Orbifold open/closed correspondence and mirror symmetry,''
{\tt arXiv:2210.11721}.

\bibitem{LZ16} W. Luo, S. Zhu,
``Integrality structures in topological strings I: framed unknot,'' 
\texttt{arXiv:1611.06506}.

\bibitem{LZ19} W. Luo, S. Zhu,
``Integrality of the LMOV invariants for framed unknot,'' 
Commun. Number Theory Phys. {\bf 13} (2019), no. 1, 81--100.

\bibitem{MV02} M. Mari\~no, C. Vafa,
``Framed knots at large $N$,'' 
Contemp. Math. {\bf 310} (2002), 185--204.

\bibitem{MNOP06} D. Maulik, N. Nekrasov, A. Okounkov, R. Pandharipande, 
``Gromov-Witten theory and Donaldson-Thomas theory, I,'' 
Compos. Math. {\bf 142} (2006), no. 5, 1263--1285.

\bibitem{MOOP11} D. Maulik, A. Oblomkov, A. Okounkov, R. Pandharipande, 
``Gromov-Witten/Donaldson-Thomas correspondence for toric 3-folds,'' 
Invent. Math. {\bf 186} (2011), 435--479.

\bibitem{MT18} D. Maulik, Y. Toda, 
``Gopakumar-Vafa invariants via vanishing cycles,'' 
Invent. Math. {\bf 213} (2018), 1017--1097.

\bibitem{Mayr01} P. Mayr, ``$N = 1$ mirror symmetry and open/closed string duality,'' \texttt{arXiv:hep-th/0108229}.

\bibitem{ORV03} A. Okounkov, N. Reshetikhin, C. Vafa, 
``Quantum Calabi-Yau and classical crystals,'' 
in {\em The unity of mathematics}, 597--618, Progress in Mathematics {\bf 244}, Birkh\"auser, Boston, 2006.

\bibitem{OV00} H. Ooguri, C. Vafa,
``Knot invariants and topological strings,'' 
Nucl. Phys. {\bf B557} (2000), 419--438.

\bibitem{PSW08} R. Pandharipande, J. Solomon, J. Walcher,
``Disk enumeration on the quintic 3-fold,''
J. Amer. Math. Soc. {\bf 21} (2008), 1169--1209.

\bibitem{PT10} R. Pandharipande, R. Thomas, 
``Stable pairs and BPS invariants,'' 
J. Amer. Math. Soc. {\bf 23} (2010), no. 1, 267--297.

\bibitem{PS19} M. Panfil, P. Su\l{}kowski,
``Topological strings, strips and quivers,'' 
JHEP {\bf 2019} (2019), 124.

\bibitem{Peng07} P. Peng,
``A simple proof of Gopakumar-Vafa conjecture for local toric Calabi-Yau manifolds,''
Commun. Math. Phys. {\bf 276} (2007), 551--569.

\bibitem{RS01} P. Ramadevi, T. Sarkar,
``On link invariants and topological string amplitudes,'' 
Nucl. Phys. {\bf B600} (2001), 487--511.

\bibitem{Schuler24} Y. Schuler,
``The log-open correspondence for two-component Looijenga pairs,''
{\tt arXiv:2404.15412}.

\bibitem{Witten89} E. Witten,
``Quantum field theory and the Jones polynomial,''
Commun. Math. Phys. {\bf 121} (1989), 351--399.

\bibitem{Witten95} E. Witten,
``Chern-Simons gauge theory as a string theory,''
in {\em The Floer memorial volume}, 637--678, Progress in Mathematics {\bf 133}, Birkh\"auser, Basel, 1995.

\bibitem{Zhou03} J. Zhou,
``Curve counting and instanton counting,''
{\tt arXiv:math/0311237}.


\bibitem{Zhu19} S. Zhu,
``Topological strings, quiver varieties, and Rogers-Ramanujan identities,'' 
Ramanujan J {\bf 48} (2019), 399--421.

\bibitem{Zhu22} S. Zhu,
``Integrality structures in topological strings and quantum 2-functions,'' 
JHEP {\bf 2022} (2022), 43.

\bibitem{Zinger11} A. Zinger, 
``A comparison theorem for Gromov-Witten invariants in the symplectic category,''
Adv. Math. {\bf 228} (2011), no. 1, 535--574.

\end{thebibliography}
